\documentclass[10pt,reqno]{amsart}

\usepackage{amsmath}
\usepackage{amsthm}
\usepackage{amsfonts}
\usepackage{amssymb}
\usepackage{graphicx}
\usepackage{hyperref}
\usepackage{color}
\usepackage[margin=1in]{geometry}

 \definecolor{skyblue}{rgb}{0.85,0.85,1}

\newtheorem{prop}{Proposition}[section]
\newtheorem{theorem}{Theorem}
\newtheorem{cor}[prop]{Corollary}
\newtheorem{lemma}[prop]{Lemma}

\DeclareMathOperator{\dv}{div}
\DeclareMathOperator{\tr}{Tr}

\DeclareMathOperator{\Mas}{Mas}
\DeclareMathOperator{\Mor}{Mor}
\newcommand{\bbR}{\mathbb{R}}

\newcommand{\cX}{\mathcal{X}}
\newcommand{\cH}{\mathcal{H}}
\newcommand{\cD}{\mathcal{D}}
\newcommand{\cF}{\mathcal{F}}
\newcommand{\Hh}{H^{1/2}(\partial \Omega) \oplus H^{-1/2}(\partial \Omega)}

\newcommand{\pO}{\partial \Omega}
\newcommand{\pOt}{\partial \Omega_t}
\newcommand{\p}{\partial}
\newcommand{\ra}{\rightarrow}

\begin{document}

\title[A Morse index theorem for elliptic operators]{A Morse index theorem for elliptic operators on bounded domains}

\author{Graham Cox}\email{ghcox@email.unc.edu}%\address{Department of Mathematics, UNC Chapel Hill, Phillips Hall CB \#3250, Chapel Hill, NC 27599}
\author{Christopher K.R.T. Jones}\email{ckrtj@email.unc.edu}%\address{Department of Mathematics, UNC Chapel Hill, Phillips Hall CB \#3250, Chapel Hill, NC 27599}
\author{Jeremy L. Marzuola}\email{marzuola@email.unc.edu}\address{Department of Mathematics, UNC Chapel Hill, Phillips Hall CB \#3250, Chapel Hill, NC 27599}

\begin{abstract}
%new version
Given a selfadjoint, elliptic operator $L$, one would like to know how the spectrum changes as the spatial domain $\Omega \subset \bbR^d$ is deformed. For a family of domains $\{\Omega_t\}_{t\in[a,b]}$ we prove that the Morse index of $L$ on $\Omega_a$ differs from the Morse index of $L$ on $\Omega_b$ by the Maslov index of a path of Lagrangian subspaces on the boundary of $\Omega$. This is particularly useful when $\Omega_a$ is a domain for which the Morse index is known, e.g. a region with very small volume. Then the Maslov index computes the difference of Morse indices for the ``original" problem (on $\Omega_b$) and the ``simplified" problem (on $\Omega_a$). This generalizes previous multi-dimensional Morse index theorems that were only available on star-shaped domains or for Dirichlet boundary conditions. We also discuss how one can compute the Maslov index using crossing forms, and present some applications to the spectral theory of Dirichlet and Neumann boundary value problems.

%
%We consider a second-order, selfadjoint elliptic operator $L$ on a family of domains $\{\Omega_t\}_{t\in[a,b]}$, %with no assumptions on the geometry of the $\Omega_t$'s. It is shown that
%and show that the Morse index of $L$ can be equated with the Maslov index of a path in a symplectic Hilbert space defined on the boundary of $\Omega_b$. This result is valid for a wide variety of boundary conditions, including Dirichlet, Neumann and Robin.
%
%Specifically, the Maslov index of the path we define relates the Morse index of $L$ on $\Omega_b$ to the Morse index of $L$ on $\Omega_a$. This is particularly useful when $\Omega_a$ is a domain for which the spectrum is more readily understood, e.g. a region with very small volume. In other words, the Maslov index exactly computes the discrepancy between the Morse indices for the ``original" problem (on $\Omega_b$) and the ``simplified" problem (on $\Omega_a$). This generalizes previous results that were only available on star-shaped domains, or for Dirichlet boundary conditions.
%
%We then discuss how one can practically compute the Maslov index using crossing forms, and present some applications to the spectral theory of Dirichlet and Neumann boundary value problems.
%%With the Morse index theorem in hand, we will also explore several monotonicity results in spectral theory, which the authors expect to have applications in stability theory related to the study of phase transitions and reaction diffusion equations in higher dimensions.  
\end{abstract}

\maketitle

\section{Introduction}
Let $L$ be a second-order, selfadjoint elliptic operator on a bounded domain $\Omega \subset \bbR^n$. The abstract spectral theory of such operators is well understood, but it is not known in general how to relate the spectrum to underlying geometric features of either the operator or the domain. For instance, if $\bar{u}$ is a steady state for the reaction-diffusion equation $u_t + f(u) = \Delta u$,
then the linear stability of $\bar{u}$ is determined by the spectrum of $L = -\Delta + f'(\bar{u})$. The operator depends explicitly on the steady state through the potential  $f'(\bar{u})$, and it would be useful if one could relate spectral properties of $L$, such as the number of negative eigenvalues, to the structure of $\bar{u}$ and $f$.

A motivating example comes from Sturm--Liouville theory for ordinary differential equations. If $\bar{u}$ is a steady state of $u_t +f(u) = u_{xx}$, then its Morse index can be found by counting the zeros of the derivative $\bar{u}_x$. In a more geometric vein, the Morse index theorem shows that the number of unstable (length decreasing) directions in which a Riemannian geodesic can be perturbed is equal to the number of conjugate points along the geodesic \cite{M63}. This relates the index to the curvature of the manifold, which affects the existence of conjugate points in a fundamental way.

A multi-dimensional Morse index theorem was proved by Smale \cite{S65} for a selfadjoint, elliptic operator $L$ on a bounded domain, with Dirichlet boundary conditions. Assuming that the domain $\Omega$ could be deformed smoothly though a family $\{\Omega_t\}$ with $\textrm{Vol}(\Omega_t) \to 0$, Smale showed that the Morse index of $L$ equals the total number of times $t$, with multiplicity, for which the problem
\begin{align*}
	Lu = 0 \textrm{ in } \Omega_t, \ \
	u = 0 \textrm{ on } \partial \Omega_t
\end{align*} 
has a nontrivial solution. These times are analogous to conjugate points in the Riemannian case, which correspond to solutions of the Jacobi equation with Dirichlet boundary conditions. An abstract generalization of this result was given by Uhlenbeck in \cite{U73}. %---and so Smale's result can be viewed as a Morse index theorem for elliptic boundary value problems.

In \cite{A85} Arnol$'$d gave a symplectic interpretation of Sturm--Liouville theory by equating the Morse index to the Maslov index---a topological invariant assigned to a path of Lagrangian subspaces in a symplectic vector space. This interpretation was extended to the multi-dimensional setting by Deng and Jones \cite{DJ11} for a Schr\"{o}dinger operator $L = -\Delta + V$ on a bounded, star-shaped domain $\Omega \subset \bbR^n$. Their idea was to contract $\Omega$ through the one-parameter family $\Omega_t := \{tx: x \in \Omega\}$, then for each $t\in(0,1]$ define a pair of Lagrangian subspaces in $\Hh$ that encode the given boundary condition and the boundary data of weak solutions to $Lu=0$ on $\Omega_t$, respectively. By construction, these subspaces intersect when there is a nonzero solution to $Lu=0$, with the prescribed boundary conditions, on $\Omega_t$. This fact was used to relate the Maslov index of the path obtained by contracting $\Omega$ to the Morse index of $L$.

In the star-shaped case the approach of Deng and Jones recovers Smale's result, but also allows one to consider more general boundary conditions. This generalization is significant because eigenvalues for a general boundary value problem can exhibit more complicated behavior, with respect to domain variations, than in the Dirichlet case. For instance, in the Neumann problem the eigenvalues are not necessarily increasing for a shrinking family of domains, as was recently observed in \cite{NW07}. 

The main shortcoming of \cite{DJ11} is the star-shaped assumption on the domain. Stability problems on general domains are of great interest, and one needs effective tools for computing the Morse index. There is also a more subtle (and important) reason for considering general domains. If $L\bar{u} = 0$, then it is desirable to relate the Morse index of $L$ to the geometric structure of $\bar{u}$, analogous to the Sturm oscillation theorem and Courant's nodal domain theorem. A relevant family of domains is given by the sublevel sets
\[
	\Omega_t = \{x \in \Omega : \bar{u}(x) < t \},
\]
which remain diffeomorphic as long as $t$ does not pass through a critical value of $\bar{u}$. There is no reason to expect the $\Omega_t$ to be star-shaped, even when $\Omega \subset \bbR^n$ is a ball and the coefficients of $L$ are radially symmetric.

In the current paper we show that, through a careful scaling of the operators and boundary conditions, it is possible to preserve the symplectic structure on the boundary as the domain is deformed, with no assumptions on the geometry of $\Omega$.  This allows us to define the Maslov index---a signed enumeration of conjugate times---and relate it to the Morse index of the boundary value problem on $\Omega$. For a family of domains $\{\Omega_t\}_{a \leq t \leq b}$, our main result is that the difference in Morse indices
\[
	\Mor(L|_{\Omega_a}) - \Mor(L|_{\Omega_b})
\]
equals the Maslov index of a path of Lagrangian subspaces in $\Hh$. We describe how to compute the relevant Maslov index in practice, and use the resulting formulas to determine Morse indices for a variety of boundary value problems.

\subsection*{Outline of the paper}
In Section \ref{sec:definitions} we make precise our assumptions on the domains, operators and boundary conditions under consideration;  the main results are stated in Section \ref{sec:results}. The path for which the Maslov index will be computed is constructed in Section \ref{sec:path}, and the main theorem is proved in Section \ref{sec:proof}. In Section \ref{sec:mono} we describe the computation of the Maslov index via crossing forms and give some applications to spectral problems with Dirichlet and Neumann boundary conditions.

Appendix \ref{forms} summarizes the relation between symmetric bilinear forms and selfadjoint, unbounded operators that lies at the heart of our presentation. A review of the Fredholm--Lagrangian Grassmannian and Maslov index for symplectic Hilbert spaces is given in Appendix \ref{app:Maslov}. In Appendix \ref{app:regular} we prove some regularity results for families of bilinear forms that are are needed in Section \ref{sec:path}.

\section*{Acknowledgments} The authors wish to thank Jian Deng, Yuri Latushkin, Alessandro Portaluri, Alim Sukhtayev, Nils Waterstraat and Kevin Zumbrun for very helpful conversations during the preparation of this manuscript. JLM was supported in part by U.S. NSF Grant DMS--1312874.  GC and CKRTJ were supported by U.S. NSF Grant DMS--1312906.

%%%%%%%%%%%%%%%%%%%%%%%%%%%%%%%%%%%%%%%%
%%%%%%%%%%%%%%%%%%%%%%%%%%%%%%%%%%%%%%%%
%%%%%%%%%%%%%%%%%%%%%%%%%%%%%%%%%%%%%%%%
%%%%%%%%%%%%%%%%%%%%%%%%%%%%%%%%%%%%%%%%
%%%%%%%%%%%%%%%%%%%%%%%%%%%%%%%%%%%%%%%%

\section{Definitions and statement of results}
\label{sec:definitions}

\subsection{The Morse index}
Throughout we assume that $\Omega \subset \mathbb{R}^n$ is a bounded domain with Lipschitz boundary. Let $L$ be a strongly elliptic operator of the form
\begin{align}
	Lu = -\partial_i (a^{ij} \partial_j u) + cu \label{eqn:L}
\end{align}
where $a^{ij}, c \in L^\infty(\Omega)$ are real-valued functions with $a^{ij} = a^{ji}$. Suppose $D$ is a Dirichlet form for $L$, i.e. a symmetric, bilinear form such that
\[
	D(u,v) = \left<Lu,v\right>_{L^2(\Omega)}
\]
for all $u,v \in C^{\infty}_0(\Omega)$. Letting $\cX$ be a closed subspace of $H^1(\Omega)$ that contains $H^1_0(\Omega)$, we say that $u \in \cX$ is an eigenfunction for the $(D,\cX)$ problem, with eigenvalue $\lambda$, if
\[
	D(u,v) = \lambda \left<u,v\right>_{L^2(\Omega)}
\]
for all $v \in \cX$. The correspondence between $D$ and $L$ is standard (see \cite{F95,K76,M00,RS72} or Appendix \ref{forms} for details). Before proceeding, we define
\begin{equation} \label{def:gamma}
	\gamma u = u|_{\partial \Omega}
\end{equation}
to be the Dirichlet trace operator, the mapping properties of which will be recalled in Lemma \ref{lemma:trprop}.

\begin{prop} \label{elliptic}
There exists an unbounded, selfadjoint operator $L_{\cX}$, with dense domain $\cD(L_{\cX}) \subset \cX$, such that
\begin{align*}
	D(u,v) = \left<L_{\cX} u,v\right>_{L^2(\Omega)}
\end{align*}
for all $u \in \cD(L_{\cX})$ and $v \in \cX$, and a first-order differential operator $B$ defined near $\pO$ such that
\[
	D(u,v) = \left<Lu,v\right>_{L^2(\Omega)} + \int_{\partial \Omega} (Bu)(\gamma v) d\mu
\]
whenever $u,v \in H^1(\Omega)$ and $Lu \in L^2(\Omega)$. Moreover, there exists an orthonormal basis for $L^2(\Omega)$ consisting of eigenfunctions $\{u_i\}$ for $L_{\cX}$, with discrete eigenvalues $\{\lambda_i\}$ tending to $\infty$.
\end{prop}

Without further regularity assumptions on $\pO$ and $D$, the eigenfunctions are only known to be in $H^1(\Omega)$. It is proved in Appendix \ref{forms} that
\[
	\cD(L_{\cX}) = \left\{ u \in \cX : Lu \in L^2(\Omega) \text{ and } \int_{\pO} (Bu)(\gamma v) d\mu = 0 \text{ for all } v \in \cX \right\}.
\]
The eigenvalues of $L_{\cX}$ satisfy the minimax principle (cf. Theorem XIII.2 in \cite{RS78})
\[
	\lambda_n = \sup_{\overset{V \subset L^2(\Omega)}{\dim(V) = n}} \inf \left\{ \frac{D(u,u)}{\|u\|_{L^2(\Omega)}^2} : u \in \cX \cap V^{\bot} \right\}
\]
and the Morse index of $L_{\cX}$ can be computed as
\[
	\Mor(L_{\cX}) = \sup \{ \dim(U) : U \subset \cX, D(u,u) < 0 \text{ for all } u \in U\}.
\]

The boundary operator $B$ depends on $D$ but not on $\cX$. The boundary conditions, and hence the domain of $L_{\cX}$, typically depend on both $D$ and $\cX$. To illustrate this dependence, we consider the form
\begin{align} \label{Dexample}
	D(u,v) = \int_{\Omega} [\nabla u \cdot \nabla v + V uv]
\end{align}
on the following closed subspaces of $H^1(\Omega)$
\begin{align*}
	 & \hspace{0cm}  \cX^0 = H^1_0(\Omega), \\
	  & \hspace{.5cm} \cX^1 = H^1(\Omega), \\
	 & \hspace{1cm}  \cX^2 = \left\{u \in H^1(\Omega) : \left.u\right|_{\Sigma_i} \text{ is constant for each } i \right\}, \\
	 & \hspace{1.5cm}  \cX^3 = \left\{u \in H^1(\Omega) : \int_{\Sigma_i} (\gamma u)  d\mu= 0 \text{ for each } i \right\},
\end{align*}
where $\{\Sigma_i\}$ are the connected components of $\pO$ and $d\mu$ is the induced volume form on $\pO$. Integrating by parts, we obtain $L = -\Delta + V(x)$ and
\[
	Bu = \left. \frac{\partial u}{\partial N}\right|_{\partial \Omega}.
\]
The selfadjoint operators $L_{{\cX}^0}, \ldots, L_{{\cX}^3}$ given by Proposition \ref{elliptic} have domains
%\begin{align*}
%	& \hspace{0cm}  \cD(L_{{\cX}^0}) = H^2(\Omega) \cap H^1_0(\Omega) , \\
%	&  \hspace{.5cm}  \cD(L_{{\cX}^1}) = \left\{u \in H^2(\Omega) : \left. \frac{\partial u}{\partial n}\right|_{\partial \Omega} = 0 \right\}, \\
%	& \hspace{1cm}  \cD(L_{{\cX}^2}) = \left\{u \in H^2(\Omega) : \left.u\right|_{\Sigma_i} \text{ is constant and } \int_{\Sigma_i} \frac{\partial u}{\partial N} \ d\mu= 0 \text{ for each } i \right\}  ,  \\
%	& \hspace{2cm}  \cD(L_{{\cX}^3}) = \left\{ u \in H^2(\Omega) : \int_{\Sigma_i} (\gamma u) d\mu = 0 \text{ and } \left. \frac{\partial u}{\partial N}\right|_{\partial \Sigma_i} \text{ is constant for each } i \right\}.
%\end{align*}
\begin{align*}
	& \hspace{0cm}  \cD(L_{{\cX}^0}) = \left\{u \in H^1(\Omega) : \Delta u \in L^2(\Omega) \text{ and } \left.u\right|_{\pO} = 0 \right\} , \\
	&  \hspace{.5cm}  \cD(L_{{\cX}^1}) = \left\{u \in H^1(\Omega) : \Delta u \in L^2(\Omega) \text{ and } \left. \frac{\partial u}{\partial n}\right|_{\partial \Omega} = 0 \right\}, \\
	& \hspace{1cm}  \cD(L_{{\cX}^2}) = \left\{u \in H^1(\Omega) :\Delta u \in L^2(\Omega),  \left.u\right|_{\Sigma_i} \text{ is constant and } \int_{\Sigma_i} \frac{\partial u}{\partial N} \ d\mu= 0 \text{ for each } i \right\}  ,  \\
	& \hspace{2cm}  \cD(L_{{\cX}^3}) = \left\{ u \in H^1(\Omega) : \Delta u \in L^2(\Omega), \int_{\Sigma_i} (\gamma u) d\mu = 0 \text{ and } \left. \frac{\partial u}{\partial N}\right|_{\partial \Sigma_i} \text{ is constant for each } i \right\}
\end{align*}
and satisfy $L_{\cX^i}u = Lu$ for $u \in \cD(L_{\cX^i})$. Without further assumptions on $\pO$ and $V(x)$ (cf. Theorem 4.18 of \cite{M00}) we cannot conclude that $\cD(L_{\cX^i}) \subset H^2(\Omega)$.

Note that $L_{{\cX^0}}$ and $L_{{\cX}^1}$ are the Dirichlet and Neumann Laplacian, respectively. The $\cX^2$ boundary conditions arise in the study of inviscid fluid flow on a multiply-connected domain---see Section 5 of \cite{L04}. One can also represent Robin boundary conditions through appropriate choices of $D$ and $\cX$; the reader is referred to \cite{F95} for further examples.

\subsection{Scaling of domains}\label{sec:scaling}
Now suppose $\{\Omega_t\}_{a\leq t \leq b}$ is a family of domains given by Lipschitz diffeomorphisms $\varphi_t\colon \Omega \rightarrow \Omega_t$. For each $t$ let $D\varphi_t\colon \Omega \to \bbR^{n\times n}$ denote the Jacobian of $\varphi_t$, which is contained in $L^\infty(\Omega,\bbR^{n\times n})$ as a consequence of Rademacher's theorem. We say that $\{\varphi_t\}$ is of class $C^k$ if $t \mapsto \varphi_t$ is in $C^k \left([a,b], L^\infty(\Omega,\bbR^n)\right)$ and $t \mapsto D\varphi_t$ is in $C^k \left([a,b], L^\infty(\Omega,\bbR^{n\times n})\right)$.

For instance, if $\Omega$ is star-shaped, we can define $\Omega_t = \{tx : x \in \Omega\}$ and $\varphi_t(x) = tx$ for $t \in [\epsilon,1]$. Another example comes from the gradient flow of a Morse function $f$. If $f^{-1}[a,b] \subset \bbR^n$ is compact and contains no critical points, it is easy to construct a family $\{\varphi_t\}$ such that $\varphi_t(\Omega) =  f^{-1}(-\infty,t]$ for $t \in [a,b]$.

It will be assumed that the Dirichlet form $D$ is defined on a domain in $\bbR^n$ that contains $\cup_{a \leq t \leq b} \Omega_t$. The above examples both satisfy $\Omega_{t_1} \subset \Omega_{t_2}$ for $t_1 < t_2$, in which case it suffices to have $D$ defined on $\Omega_b$. We define a family of Dirichlet forms $\{D_t\}$ on $\cX \subset H^1(\Omega)$ by
\begin{align}\label{Dtdef}
	D_t(u,v) = \left.D\right|_{\Omega_t}(u \circ \varphi_t^{-1}, v \circ \varphi_t^{-1}).
\end{align}
Each $D_t$ is symmetric and coercive, so by Proposition \ref{elliptic} there exists a family of unbounded, selfadjoint operators $\{L_{\cX,t}\}$ on $L^2(\Omega)$ such that $D_t(u,v) = \left<L_{\cX,t}u,v\right>_{L^2(\Omega)}$ for each $u \in \cD(L_{\cX,t})$ and $v \in \cX$, and operators $L_t$ and $B_t$ such that
\begin{align} \label{LtBt}
	D_t(u,v) = \left< L_t u, v \right>_{L^2(\Omega)} + \int_{\pO} (B_tu)(\gamma v) d\mu
\end{align}
whenever $u,v \in H^1(\Omega)$ and $L_t u \in L^2(\Omega)$.

Our main result, Theorem \ref{thm:main}, relates the Morse indices of $\{L_{\cX,a}\}$ and $\{L_{\cX,b}\}$. Both operators are defined on $L^2(\Omega)$. It follows from a change of variables that the $(D_t,\cX)$ eigenvalue problem is equivalent to the $( D|_{\Omega_t}, \cX_t)$ problem, where $\cX_t := \{u \circ \varphi_t^{-1} : u \in \cX\} \subset H^1(\Omega_t)$. 
%In particular, there exists $u \in \cX$ such that $D_t(u,v) = 0$ for all $v \in \cX$ if and only if there exists $u_t \in \cX_t$ such that $D|_{\Omega_t}(u_t,v_t) = 0$ for all $v_t \in \cX_t$. It follows that $u_t$ satisfies the equation $L u_t = 0$ strongly in $\Omega_t$.
To determine the boundary conditions on $\pOt$, it is necessary to identify $\cX_t$ explicitly. For the examples considered above we have
\begin{align*}
	& \hspace{0cm}  \cX^0_t = H^1_0(\Omega_t) , \\
	& \hspace{0.5cm}\cX^1_t = H^1(\Omega_t) , \\
	& \hspace{1cm}  \cX^2_t = \left\{u \in H^1(\Omega_t) : \left.u\right|_{\Sigma_{t_i}} \text{ is constant for each } i \right\} , \\
	& \hspace{1.5cm} \cX^3_t = \left\{u \in H^1(\Omega_t) : \int_{\Sigma_{t_i}} (\gamma_t u) (\varphi_t^{-1})^*d\mu= 0 \text{ for each } i \right\},
\end{align*}
where $\gamma_t$ denotes the Dirichlet trace on $\Omega_t$. In the first three cases $\cX^j_t$ depends on $\Omega_t$, but not the particular diffeomorphism $\varphi_t\colon \Omega \rightarrow \Omega_t$. On the other hand,  $\cX^3_t$ is not, in general, equal to the space
\[
	\left\{u \in H^1(\Omega_t) : \int_{\Sigma_{t_i}} (\gamma_t u) d\mu_t= 0 \text{ for each } i \right\},
\]
because the pulled-back volume form $(\varphi_t^{-1})^*d\mu$ on $\pOt$ does not necessarily agree with the induced form $d\mu_t$. Therefore the interpretation of a conjugate time---a value of $t$ for which the $( D|_{\Omega_t}, \cX^3_t)$ problem has a nontrivial kernel---depends on the diffeomorphisms $\{\varphi_t\}$ and not just the family of domains $\{\Omega_t\}$.

One can always modify $\{\varphi_t\}$ to obtain a new family $\{\widehat{\varphi}_t\}$ such that $\widehat{\varphi}_t(\pO) = \varphi_t(\pO)$ for all $t$, and
\[
	\widehat{\cX}^3_t = \left\{u \in H^1(\Omega_t) : \int_{\Sigma_{ti}} (\gamma_t u) d\mu_t= 0 \text{ for each } i \right\},
\]
but we will not explore this issue any further in the current paper.

\subsection{A symplectic Hilbert space}
We define
\[
	\cH = \Hh.
\]
In Appendix \ref{app:Maslov} it is shown that $\cH$ has the structure of a symplectic Hilbert space. Through a minor abuse of notation, we will denote the dual pairing between $H^{1/2}(\pO)$ and $H^{1/2}(\pO)^* \cong H^{-1/2}(\pO)$ by the integral notation
\[
	{}_{H^{1/2}(\pO)}\left< f, g\right>_{H^{-1/2}(\pO)} = \int_{\pO} fg \,d\mu
\]
for $f \in H^{1/2}(\pO)$ and $g \in H^{-1/2}(\pO)$.

We now construct two families of Lagrangian subspaces of $\cH$, corresponding to the rescaled differential operators and boundary conditions, respectively. The space of weak solutions to $L_t u = \lambda u$, in the absence of boundary conditions, is denoted by
\begin{align} \label{Klt}
	K_{\lambda,t} = \left\{u \in H^1(\Omega) : D_t(u,v) = \lambda \left<u,v\right>_{L^2(\Omega)} \text{ for all } v \in H^1_0(\Omega) \right\}
\end{align}
for $(\lambda,t) \in \bbR \times [a,b]$. We define a trace map $\tr_t\colon C^1(\overline{\Omega}) \rightarrow C^0(\pO) \times C^0(\pO)$ by
\begin{align} \label{def:tr}
	\tr_t u = \left( \gamma u, B_t u\right),
\end{align}
where $\gamma$ is the Dirichlet trace operator from \eqref{def:gamma} and $B_t$ is the rescaled boundary operator from \eqref{LtBt}. It is observed in Lemma \ref{lemma:trprop} that $\tr_t$ extends to a bounded operator on $K_{\lambda,t}$, so we can define
\begin{align} \label{mudef}
	\mu(\lambda,t) = \tr_t(K_{\lambda,t}).
\end{align}
We also define the space of admissible boundary values by
\begin{align} \label{boundary}
	\nu = \left\{(f,g) \in \cH : f \in \gamma(\cX), \int_{\pO} g (\gamma v) d\mu = 0 \text{ for all } v \in \cX \right\}.
\end{align}

Again referring to the four examples above, we have
\begin{align*}
	& \hspace{0cm}  \nu^0 = \{0\} \oplus H^{-1/2}(\pO) , \\
	& \hspace{0.5cm} \nu^1 = H^{1/2}(\pO) \oplus \{0\} , \\
	& \hspace{1cm}  \nu^2 = \left\{ (f,g) \in \cH : \left.f\right|_{\Sigma_i} \text{ is constant and } \int_{\Sigma_i} g \ d\mu= 0 \text{ for each } i \right\} , \\
	& \hspace{1.5cm}  \nu^3 = \left\{ (f,g) \in \cH :  \int_{\Sigma_i} f \ d\mu= 0 \text{ and } \left.g \right|_{\Sigma_i} \text{ is constant for each } i \right\}.
\end{align*}

\subsection{Conjugate times}
The spaces $\mu(\lambda,t)$ and $\nu$ are defined so a nontrivial intersection corresponds to an eigenvalue of $L_{\cX,t}$, as we prove in Section \ref{ss:int}.

\begin{prop} \label{intersection}
The intersection $\mu(\lambda,t) \cap \nu$ is nontrivial if and only if there is a nonzero function $u \in \cD(L_{\cX,t})$ with $L_{\cX,t} u = \lambda u$. Moreover,
\[
	\dim \left[ \mu(\lambda,t) \cap \nu \right] = \dim \ker \left( L_{\cX,t} - \lambda \right).
\]
\end{prop}

We say that $t_* \in [a,b]$ is a \emph{conjugate time} if $\mu(0,t_*) \cap \nu \neq \{0\}$. Thus $t_*$ is a conjugate time if and only if $L_{\cX,t_*}$ has a nontrivial kernel, which is true if and only if
\[
	\ker D_{t_*} := \{u \in \cX : D_{t_*}(u,v) = 0 \text{ for all } v \in \cX \}
\]
is nontrivial. By a change of coordinates we see that $\ker D_{t_*}$ is isomorphic to
\[
	\ker \left.D\right|_{\Omega_{t_*}} := \{u \in \cX_{t_*} : \left.D\right|_{\Omega_{t_*}}(u,v) = 0 \text{ for all } v \in \cX_{t_*} \}.
\]

For our example \eqref{Dexample}, $t_*\in [a,b]$ is a conjugate time for the $\cX^0$ (Dirichlet) problem if there exists $u \in H^1(\Omega_{t_*})$ such that
\begin{align*}
	-\Delta u + V(x) u = 0, \ \left. u \right|_{\partial \Omega_{t_*}} = 0,
\end{align*}
and is a conjugate time for the $\cX^1$ (Neumann) problem if there exists $u \in H^1(\Omega_{t_*})$ such that
\begin{align*}
	-\Delta u + V(x) u = 0, \ \left. \frac{\partial u}{\partial N_{t_*}} \right|_{\partial \Omega_{t_*}} = 0.
\end{align*}

Analogous to \eqref{LtBt}, there is an operator $\widehat{B}_t$ such that
\begin{align} \label{Bhatdef}
	\left.D\right|_{\Omega_{t}}(u,v) = \left< Lu, v \right>_{L^2(\Omega_t)} + \int_{\pOt} (\widehat{B}_t u)(\gamma v) d\mu
\end{align}
whenever $u,v \in H^1(\Omega_t)$ and $Lu \in L^2(\Omega_t)$. In the example above, $\widehat{B}_t = \partial / \partial N_t$ on $\pOt$, whereas the rescaled boundary operator $B_t$ on $\pO$ is given by a more complicated expression involving the Jacobian of $\varphi_t$.

\subsection{Statement of results} \label{sec:results}
By construction, $\{\mu(\lambda,t)\}$ is a smooth family of Lagrangian subspaces and has a well-defined Maslov index with respect to $\nu$. Our main result is the following.

\begin{theorem} \label{thm:main}
Let $\Omega \subset \bbR^n$ be a bounded domain with Lipschitz boundary, and $\varphi_t\colon \Omega \ra \Omega_t$ a $C^0$ family of Lipschitz diffeomorphisms for $t \in [a,b]$ (as defined in Section \ref{sec:scaling}). Suppose $D$ is a strongly elliptic Dirichlet form with continuous coefficients, and $\cX \subset H^1(\Omega)$ is a closed subspace that contains $H^1_0(\Omega)$. With $L_{\cX,t}$, $\mu(\lambda,t)$ and $\nu$ defined as above, the Maslov index of $\mu(\lambda,t)$ with respect to $\nu$ satisfies
%\begin{align} \label{eqn:MM}
%	\Mas(\left. \mu(0,t) \right|_{a \leq t \leq b}; \nu) = \Mor(L_{\cX,a}) - \Mor(L_{\cX,b}).
%\end{align}
\begin{align} \label{eqn:MM}
	\Mas( \mu(0,t) ; \nu) = \Mor(L_{\cX,a}) - \Mor(L_{\cX,b}).
\end{align}
\end{theorem}

The Maslov index gives a signed count of the conjugate times in $[a,b]$, and it is natural to ask when the difference in Morse indices is in fact equal to the number of conjugate times. This requires monotonicity of the Maslov index, in the sense that all intersections of $\mu(0,t)$ and $\nu$ have the same orientation. This is easily shown for Dirichlet problem when the domains and operators are sufficiently regular.

\begin{cor} \label{cor:Dir}
Additionally assume that the family $\{\varphi_t\}$ is $C^1$, each $\pOt$ is of class $C^{1,1}$, and the coefficients of $D$ are continuous differentiable. Let $\cX = H^1_0(\Omega)$, so that $L_{\cX,t} = L_{D,t}$ is the Dirichlet realization of $L$. If $\Omega_{t_1} \subset \Omega_{t_2}$ for $t_1 < t_2$, then the number of conjugate times in $[a,b]$ is finite and
\begin{align}
	\Mor(L_{D,b}) = \Mor(L_{D,a}) + \sum_{t \in [a,b)} \dim \ker D_t.
\end{align}
\end{cor}

This is precisely the index theorem proved by Smale in \cite{S65}. A symplectic interpretation was given by Swanson in \cite{S78II}; our method differs in its ability to handle more general boundary conditions.
Note that the sum includes $t=a$ but not $t=b$, so it is not relevant if $b$ is a conjugate time. Intuitively, this is because the Dirichlet spectrum is strictly decreasing with respect to $t$, so an eigenvalue that equals zero at $t=b$ is positive for $t < b$ and hence does not contribute to the Morse index. For general boundary conditions an intersection at $t=b$ can only contribute nonpositively to the Morse index.

While such monotonicity cannot always be expected, one use crossing forms (defined in Appendix \ref{app:Maslov}) to determine the direction of intersection between $\mu$ and $\nu$ and hence find the contribution to the Morse index from each conjugate time. A conjugate time corresponds to a zero eigenvalue for $D_t$, with multiplicity $\dim \ker D_t$; the crossing form determines how many of these eigenvalues are increasing, and how many are decreasing, with respect to $t$. For related results on the motion of simple eigenvalues see \cite{BW2014,GS53,H05} and references therein.

In the star-shaped case, where $\Omega_t := \{tx : x \in \Omega\}$ for $t \in (0,1]$, the rescaled Dirichlet form $D_t$ can be computed easily, and one obtains more explicit expressions for the crossing form than are generally available. In particular, it is possible to deduce monotonicity results for the \emph{Neumann Laplacian} $-\Delta_N$, which we define to be the unbounded, selfadjoint operator corresponding to the Dirichlet form $D(u,v) = \int_{\Omega} \nabla u \cdot \nabla v$ with domain $\cX = H^1(\Omega)$, and similarly for the rescaled operators $-\Delta_{N,t}$ on $\Omega_t$. %For ease of exposition we assume here (as in the rest of the paper) that $V(x)$ is smooth on $\overline{\Omega}$.

\begin{cor} \label{cor:Neumannpotential}
Let $\Omega \subset \bbR^n$ be a star-shaped domain with $C^{1,1}$ boundary. Suppose $V \in C^1(\overline{\Omega})$ and $\lambda$ is an eigenvalue of multiplicity $k$ for $L_{N,t} := -\Delta_{N,t}+V(x)$ for some $t \in (0,1)$. If
\begin{align} \label{biglambda}
	\lambda > V(x) + \frac{1}{2} x \cdot \nabla V(x)
\end{align}
for all $x \in \Omega_t$, then
\[
	\Mor(L_{N,t+\delta}-\lambda) = \Mor(L_{N,t-\delta}-\lambda) + k
\]
for $\delta > 0$ sufficiently small.
\end{cor}

In other words, as the domain expands from $\Omega_{t - \delta}$ to $\Omega_{t + \delta}$, the number of Neumann eigenvalues below $\lambda$ increases by $k$, assuming $\lambda$ is sufficiently large. Setting $V = 0$ we find that any positive eigenvalue of the Neumann Laplacian satisfies
\begin{align}
	\Mor(-\Delta_{N,t+\delta}-\lambda) = \Mor(-\Delta_{N,t-\delta}-\lambda) + k,
\end{align}
under the hypotheses of Corollary \ref{cor:Neumannpotential}. While seemingly elementary, this result is actually rather subtle, because the monotonicity of the eigenvalues (or Morse index) for the Neumann Laplacian is known to fail for domains that are not star-shaped, even in the radially symmetric case. For instance, it was shown in \cite{NW07} that the first nonzero Neumann eigenvalue on the annulus
\[
	A_{r,R} := \{x \in \bbR^n : r \leq |x| \leq R\}
\]
is decreasing with respect to both $r$ and $R$. This differs from the first Dirichlet eigenvalue, which is decreasing in $R$ but increasing in $r$.

%The result of Corollary \ref{cor:Neumann} does not immediately generalize to the Neumann problem for a Schr\"odinger operator $-\Delta_N + V$. However, we find that a similarly monotonicity property holds for sufficiently large eigenvalues. 
%Thus the conclusion of Corollary \ref{cor:Neumann} holds if $\lambda$ additionally satisfies \eqref{biglambda}. 

By a unique continuation argument it suffices to have
\[
	\lambda \geq V(x) + \frac{1}{2} x \cdot \nabla V(x)
\]
for all $x$, with strict inequality on a nonempty, open subset of $\Omega_t$. Since $V$ and $\nabla V$ are bounded on $\Omega$, there are only a finite number of eigenvalues for which this condition could fail. If the potential is radial, $V(x) = f(|x|)$, this is equivalent to
\[
	\lambda \geq f(r) + \frac{r }{2} f'(r) %= \frac{1}{2r} \frac{d}{dr} \left[ r^2 f(r) \right]
\]
for $r \leq t$.
%In particular, if $f'(r) \leq 0$, then \eqref{biglambda} is satisfied for any $\lambda > f(0)$. (As above, it suffices to have $\lambda \geq f(0)$, unless $f(r)$ is constant.)

As a final example, suppose the potential satisfies 
\[
	0 > V(x) + \frac{1}{2} x \cdot \nabla V(x)
\]
for all $x \in \Omega$ (which in particular implies $V(0) < 0$). Then the Morse index of  $-\Delta_N + V(x)$ can be related to the number of conjugate times $t \in (0,1)$, as in Corollary \ref{cor:Dir}. Letting $c(t)$ denote the dimension of the solution space of
\[
	-\Delta u + V(x) u \text{ in } \Omega_t, \ \ \frac{\partial u}{\partial N_t} = 0 \text{ on } \pOt
\]
for each $t \in (0,1)$, we have
\begin{align}
	\Mor(-\Delta_N + V) = \sum_{t \in (0,1)} c(t) + 1.
\end{align}

%%%%%%%%%%%%%%%%%%%%%%%%%%%%%%%%%%%%%%%%
%%%%%%%%%%%%%%%%%%%%%%%%%%%%%%%%%%%%%%%%
%%%%%%%%%%%%%%%%%%%%%%%%%%%%%%%%%%%%%%%%
%%%%%%%%%%%%%%%%%%%%%%%%%%%%%%%%%%%%%%%%
%%%%%%%%%%%%%%%%%%%%%%%%%%%%%%%%%%%%%%%%

\section{Construction of the symplectic path}
\label{sec:path}
We now give in detail the construction of the subspaces $\mu(\lambda,t)$ and $\nu$ outlined in Section \ref{sec:definitions}. Throughout we consider the symplectic Hilbert space $\cH := \Hh$ with symplectic form $\omega$ defined by
\begin{align}
	\omega \left( (f_1,g_1), (f_2,g_2) \right) = \int_{\pO} (f_1 g_2 - f_2 g_1) d\mu,
\end{align}
where $d\mu$ denotes the induced area form on $\pO$. We denote by $J\colon \cH \ra \cH$ the almost complex structure on $\cH$, given by
\begin{equation}
\label{eqn:Jdef}	
J(f,g) = \left(R^{-1}g, -Rf \right)
\end{equation}
for $(f,g) \in \cH$, where $R\colon H^{1/2}(\pO) \ra H^{-1/2}(\pO) \cong H^{1/2}(\pO)^*$ is the Riesz duality isomorphism.

The main definitions and properties of symplectic Hilbert spaces are given in Appendix \ref{app:Maslov}; for now we simply recall that $\Lambda(\cH)$ denotes the Lagrangian Grassmannian of $\cH$ and $\cF \Lambda_{\nu}(\cH)$ denotes the Fredholm--Lagrangian Grassmannian with respect to a fixed Lagrangian subspace $\nu \in \Lambda(\cH)$. The following proposition, the main result of this section, summarizes the properties of $\mu(\lambda,t)$ and $\nu$ needed in the proof of Theorem \ref{thm:main}.

\begin{prop} If the hypotheses of Theorem \ref{thm:main} are satisfied, then % $\mu \in C^{\infty}\left( \bbR \times [a,b]; \cF_{\nu}(\Lambda(\cH)) \right)$.
$\mu(\cdot,t_0) \in C^{\infty}\left( \bbR, \cF \Lambda_{\nu}(\cH) \right)$ and $\mu(\lambda_0,\cdot) \in C\left([a,b], \cF \Lambda_{\nu}(\cH) \right)$ for any fixed  $\lambda_0 \in \bbR$ and $t_0 \in [a,b]$. Moreover, if $\{\varphi_t\}$ is of class $C^k$, then $\mu(\lambda_0,\cdot) \in C^k\left([a,b], \cF \Lambda_{\nu}(\cH) \right)$.
\end{prop}

In particular, for each $(\lambda,t) \in \bbR \times [a,b]$ the subspaces $\mu(\lambda,t)$ and $\nu$ are Lagrangian and comprise a Fredholm pair. Moreover, $\mu(\lambda,t)$ is smooth in $\lambda$ and $C^k$ in $t$. As described in Appendix \ref{app:Maslov}, the Maslov index is defined for any continuous path in the Fredholm--Lagrangian Grassmannian, but its computation via crossing forms requires differentiability.

We assume for the remainder of the section that the hypotheses of Theorem \ref{thm:main} are satisfied.

\subsection{The trace map}
Recall that for each $t \in [a,b]$, there exist operators $L_t$ and $B_t$ such that
\begin{align} \label{weakGreen}
	D_t(u,v) = \left< L_t u, v \right>_{L^2(\Omega)} + \int_{\pO} (B_tu)(\gamma v) d\mu
\end{align}
provided $u,v \in H^1(\Omega)$ and $L_t u \in L^2(\Omega)$ (cf. Theorem 4.4 of \cite{M00}).

We define the space $H^{1,0}_{L_t}(\Omega) = \{ u \in H^1(\Omega) : L_t u \in L^2(\Omega) \}$
with the graph norm
\[
	\|u\|_{L_t}^2 = \|u\|_{H^1(\Omega)}^2 + \|L_t u\|_{L^2(\Omega)}^2.
\]
Note that $K_{\lambda,t} \subset H^{1,0}_{L_t}$ and each $u \in K_{\lambda,t}$ satisfies $\|u\|_{L_t} \leq C \|u\|_{H^1(\Omega)}$ for some constant $C = C(\lambda,t)$. The following lemma shows that $H^{1,0}_{L_t}(\Omega)$ is an appropriate domain for the trace operator.

\begin{lemma} \label{lemma:trprop}
For each $t \in [a,b]$ the map $\tr_t$ defined in \eqref{def:tr} extends to a bounded map
\[
	\tr_t\colon H^{1,0}_{L_t}(\Omega) \longrightarrow \Hh.
\]
%and the identity
%\begin{align} \label{weakGreen}
%	D_t(u,v) = \left<L_t u,v\right>_{L^2(\Omega)} + \int_{\pO} (B_tu)(\gamma v) d\mu
%\end{align}
%holds for all $u \in H^{1,0}_{L_t}(\Omega)$ and $v \in H^1(\Omega)$.
Moreover, if $U \subset \bbR \times [a,b]$ is open and $u_{\lambda,t} \in C^k(U, H^1(\Omega))$ satisfies $u_{\lambda,t} \in K_{\lambda,t}$ for all $(\lambda,t) \in U$, then $\tr_t(u_{\lambda,t}) \in C^k(U,\cH)$.
\end{lemma}

\begin{proof} 
The boundedness of $\tr_t$ follows from Theorem 3.37 and Lemma 4.3 of \cite{M00}.

The differentiability of $(\lambda,t) \mapsto \gamma u_{\lambda,t}$ and $\lambda \mapsto B_t u_{\lambda,t}$ follows immediately. However, the regularity of the map $t \mapsto B_t u_{\lambda,t}$ is more subtle since the domain of $B_t$ is $t$-dependent.

Since $K_{\lambda,t} \subset H^{1,0}_{L_t}$, it follows that
\begin{equation}
\label{eqn:energy}	
\int_{\pO} (B_t u_{\lambda,t}) (\gamma v) d\mu = D_t(u_{\lambda,t},v) - \lambda \left<u_{\lambda,t},v\right>_{L^2(\Omega)}
\end{equation}
for all $v \in H^1(\Omega)$ and all $(\lambda,t) \in U$.
Equivalently, 
\[
	\int_{\pO} (B_t u_{\lambda,t}) f \,d\mu = D_t(u_{\lambda,t},Ef) - \lambda \left<u_{\lambda,t},Ef\right>_{L^2(\Omega)}
\]
for all $f \in H^{1/2}(\pO)$, where $E\colon H^{1/2}(\pO) \ra H^1(\Omega)$ is a bounded right inverse for the Dirichlet trace $\gamma$.
By assumption $u_{\lambda,t} \in C^k(U,H^1(\Omega))$ and $D_t$ is smooth, so we find that $B_t u_{\lambda,t} \in C^k(U, H^{-1/2}(\pO))$.
\end{proof}

The next lemma, a consequence of the unique continuation property for second-order elliptic operators, shows that $\tr_t$ gives an isomorphism from $K_{\lambda,t}$ onto $\mu(\lambda,t)$. This implies $\dim \tr_t(V) = \dim V$ for any finite-dimensional subspace $V \subset K_{\lambda,t}$ (cf. Proposition \ref{intersection}).

\begin{lemma} \label{Carleman}
For each $(\lambda,t) \in \bbR \times [a,b]$ there exists $C = C(\lambda,t)$ such that
\begin{equation}
\label{tr-est}
	\|u\|_{H^1(\Omega)} \leq C \|\tr_t u\|_{\cH}
\end{equation}
for every $u \in K_{\lambda,t}$.
\end{lemma}

The constant $C$ can be chosen uniformly on compact subsets of $\bbR \times [a,b]$, but we do not require such generality.

\begin{proof}
 It follows from the coercivity of $D_t$ that
\begin{equation}
\label{tr-est-gen}
	\|u\|_{H^1(\Omega)} \leq C (\|\tr_t u\|_{\cH} +  \| u \|_{L^2 (\Omega)})
\end{equation}
for all $u \in K_{\lambda,t}$. To obtain the stronger estimate \eqref{tr-est} we argue by contradiction, using a standard compactness argument.

Assuming the existence of a sequence $\{u_i\}$ in $K_{\lambda,t}$ with $\|u_i\|_{L^2(\Omega)} = 1$ and $\|u_i\|_{H^1} \geq i \|\tr_t u_i\|_{\cH}$ for each $i$, we conclude from \eqref{tr-est-gen} that $\{u_i\}$ is bounded in $H^1(\Omega)$. Therefore, there is a function $\bar{u} \in H^1(\Omega)$ with $\| \bar{u} \|_{L^2(\Omega)} = 1$, and a subsequence $\{u_i\}$, such that $u_i \ra \bar{u}$ in $L^2(\Omega)$ and $u_i \rightharpoonup \bar{u}$ in $H^1(\Omega)$. It follows that $\bar{u} \in K_{\lambda,t}$, and so $\tr_t \bar{u} \in \cH$ is defined. The boundedness of $\gamma\colon H^1(\Omega) \to H^{1/2}(\pO)$ implies $\gamma u_i \rightharpoonup \gamma \bar{u}$ in $H^{1/2}(\pO)$ and \eqref{weakGreen} yields $B_t u_i \rightharpoonup B_t \bar{u}$ in $H^{-1/2}(\pO)$, hence $\tr_t u_i \rightharpoonup \tr_t \bar{u}$. Since $\{u_i\}$ is bounded in $H^1(\Omega)$ we have $\tr_t u_i \ra 0$ in $\cH$, which implies $\tr_t \bar{u} = 0$.

By construction $\bar{u} \in H^1(\Omega)$ is a nonvanishing weak solution to $L_t \bar{u} = \lambda \bar{u}$, with boundary data $\gamma \bar{u} = 0$ and $B_t \bar{u} = 0$. It follows from a unique continuation argument (see Proposition 2.5 of \cite{BR12}) that this is only possible if $\bar{u} \equiv 0$, so we obtain a contradiction and the proof is complete. 
%The required unique continuation argument is given in Proposition 2.5 of \cite{BR12}; see also Theorem 3.2.2 in \cite{isakov2006inverse}, and the general survey \cite{tataru2004unique}.
\end{proof}

\subsection{The solution space}
We now turn our attention to the space $\mu(\lambda,t) = \tr_t(K_{\lambda,t})$.

\begin{lemma} \label{muisotropic}
For each $(\lambda,t) \in \bbR \times [a,b]$, $\mu(\lambda,t)$ is a closed, isotropic subspace of $\cH$.
\end{lemma}

\begin{proof} That $\mu(\lambda,t)$ is closed in $\cH$ follows immediately from Lemmas \ref{lemma:trprop} and \ref{Carleman} and the fact that $K_{\lambda,t}$ is a closed subspace of $H^1(\Omega)$. To see that $\mu(\lambda,t)$ is isotropic, consider $u,v \in K_{\lambda,t}$. It follows from \eqref{weakGreen} that
\[
	\int_{\pO} (B_tu) (\gamma v) d\mu = \int_{\pO} (B_tv)(\gamma u) d\mu,
\]
hence $\omega(\tr_t u, \tr_t v) = 0$ as required.
\end{proof}

%%%%%%%%% the regularity proof %%%%%%%%%%%%%%%
We next analyze the regularity of $\mu(\lambda,t)$ in the Lagrangian Grassmannian, recalling that the topology on $\Lambda(\cH)$ is defined by identifying a subspace $\mu$ with the orthogonal projection $P_{\mu}$ in the space of bounded operators $B(\cH)$. If $\rho \in \Lambda(\cH)$ and $A\colon \rho \ra \rho$ is a bounded, selfadjoint operator, then the graph of $A$ over $\rho$, defined by
\[
	G_{\rho}(A) = \{x + JAx : x \in \rho\}
\]
with $J$ as in \eqref{eqn:Jdef}, is also Lagrangian. By equation (2.16) of \cite{F04} the corresponding orthogonal projection is
\begin{align} \label{projA}
	P_{G_{\rho}(A)}(x + Jy) = (I + JA) \left[ (I + A^2)^{-1}(x+Ay) \right]
\end{align}
for $x,y \in \rho$, so it suffices to express $\{\mu(\lambda,t)\}$ as the graph of a suitably smooth family $\{A(\lambda,t)\}$ of selfadjoint operators on a fixed Lagrangian subspace.

If $L_t-\lambda$ has trivial Neumann kernel, then $\mu(\lambda,t)$ is the graph of the Neumann-to-Dirichlet map over the Lagrangian subspace $\{0\} \oplus H^{-1/2}(\pO)$, which can be shown to vary smoothly in $\lambda$ and $t$. More generally, in the proof of the following proposition we show that one can always find a Robin boundary condition for which $L_t - \lambda$ is invertible, then express $\mu(\lambda,t)$ as the graph of the corresponding Robin-to-Robin map.

\begin{prop} For each $(\lambda,t) \in \bbR \times [a,b]$, $\mu(\lambda,t)$ is a Lagrangian subspace of $\cH$ and $\mu(\cdot,t) \in C^{\infty}\left( \bbR, \Lambda(\cH) \right)$. If $\{\varphi_t\}$ is of class $C^k$, then $\mu(\lambda,\cdot) \in C^k\left([a,b], \Lambda(\cH) \right)$.
\end{prop}

In the following proof (and nowhere else) a Banach space-valued map is called ``smooth" if it is $C^\infty$ with respect to $\lambda$ and $C^k$ with respect to $t$.

\begin{proof}

Fix $(\lambda_0,t_0) \in \bbR \times [a,b]$. We claim that there is an open set $U \subset \bbR \times [a,b]$ containing $(\lambda_0,t_0)$, a Lagrangian subspace $\rho \subset \cH$ and a family of bounded, selfadjoint operators $A(\lambda,t)\colon \rho \ra \rho$, such that $G_{\rho}(A(\lambda,t)) = \mu(\lambda,t)$ for $(\lambda,t) \in U$. The family $A(\cdot,\cdot)\colon U \to B(\rho)$ is smooth, so it follows from \eqref{projA} that the map $(\lambda,t) \mapsto P_{\mu(\lambda,t)}$ is smooth, completing the proof.

To see that the claimed $U$ and $A$ exist, we define a perturbed Dirichlet form $D_{\beta,\lambda,t}$ by
\[
	D_{\beta,\lambda,t}(u,v) = D_t(u,v) -\lambda \left<u,v\right>_{L^2(\Omega)} - \beta \int_{\pO} (R \gamma u)(\gamma v) d\mu
\]
for $u,v \in H^1(\Omega)$ and $\beta \in \bbR$. It follows from Theorem 3.2 of \cite{Rohleder2014} that $D_{\beta_0,\lambda_0,t_0}$ is invertible for some $\beta_0$, and Lemma \ref{openinv} implies $D_{\beta_0,\lambda,t}$ is invertible in a neighborhood $U$ of $(\lambda_0,t_0)$.

We define the subspace
\[
	\rho = \{ (f,g) \in \cH : f + \beta_0 R^{-1}  g = 0 \},
\]
which is Lagrangian, with
\[
	J \rho = \{ (f,g) \in \cH : g - \beta_0 Rf = 0 \}.
\]
Let $(f,g) \in \rho$. For each $(\lambda,t) \in U$  there exists a unique function $u_{\lambda,t} \in H^1(\Omega)$ such that
 \begin{align} \label{betaeqn}
 	D_{\beta_0,\lambda,t}(u_{\lambda,t},v) = \int_{\pO} (g - \beta_0 Rf)(\gamma v) d\mu
 \end{align}
for all $v \in H^1(\Omega)$. In particular, $D_t(u_{\lambda,t},v) = \lambda \left<u_{\lambda,t},v\right>_{L^2(\Omega)} $ for $v \in H^1_0(\Omega)$, so $u_{\lambda,t} \in K_{\lambda,t}$. Proposition \ref{Dform} implies $(\lambda,t) \mapsto u_{\lambda,t}$ is smooth in $H^1(\Omega)$ and it follows from Lemma \ref{lemma:trprop} that the path
\begin{align} \label{betatrace}
	(\lambda,t) \mapsto B_t u_{\lambda,t} = g - \beta_0 Rf + \beta_0  R \gamma (u_{\lambda,t})
\end{align}
is smooth in $H^{-1/2}(\pO)$. (For our choice of boundary conditions, the regularity of the above map only requires the boundedness of $\gamma$ and not the full statement of Lemma \ref{lemma:trprop}.)

Since $J$ is an isomorphism, we can implicitly define $A(\lambda,t)\colon \rho \ra \cH$ by
\[
	JA(\lambda,t)(f,g) = \left( \gamma (u_{\lambda,t}) - f,  \beta_0 R \gamma (u_{\lambda,t}) -\beta_0 R f \right)
\]
for $(f,g) \in \rho$. It follows that $JA(\lambda,t)(f,g) \in J \rho$, so we in fact have $A(\lambda,t)\colon \rho \ra \rho$.

To see that $A$ is selfadjoint, we take $(f_1,g_1)$ and $(f_2,g_2)$ in $\rho$, and let $u_1$ and $u_2$ denote the respective solutions  to \eqref{betaeqn} (omitting the $\lambda$ and $t$ subscripts for convenience). Writing \eqref{betaeqn} for $u_1$ with the test function $v=u_2$, and vice versa, we have
\[
	D_t(u_1,u_2) - \lambda \left<u_1,u_2\right>_{L^2(\Omega)} = \int_{\pO}[\beta_0 R(\gamma u_1 - f_1) + g_1](\gamma u_2) d\mu
\]
and
\[
	D_t(u_2,u_1) - \lambda \left<u_2,u_1\right>_{L^2(\Omega)} = \int_{\pO}[\beta_0 R(\gamma u_2 - f_2)+ g_2](\gamma u_1) d\mu.
\]
Subtracting and using the fact that $\int_{\pO} (Rh_1)h_2 d\mu = \int_{\pO} (Rh_2)h_1 d\mu$ for $h_1,h_2 \in H^{1/2}(\pO)$ yields
\[
	\int_{\pO} [ g_1 - \beta_0 R f_1](\gamma u_2) d\mu = \int_{\pO} [ g_2 - \beta_0 R f_2](\gamma u_1) d\mu.
\]
We next recall the relation $\omega(x,y) = \left<Jx,y\right>_{\cH}$ for all $x,y \in \cH$ and compute using the above equality
\begin{align*}
	 \left<A(f_1,g_1), (f_2,g_2)\right>_{\cH} - \left<A(f_2,g_2), (f_1,g_1)\right>_{\cH} 
	& = \omega \left( JA(f_2,g_2), (f_1,g_1)\right) - \omega \left( JA(f_1,g_1), (f_2,g_2)\right) \\
	&  = \int_{\pO} \left[ f_1 g_2  -\beta_0 (R f_1) f_2 - f_2 g_1 + \beta_0 (R f_2)f_1 \right] d\mu \\
	& = \omega\left( (f_1,g_1), (f_2,g_2) \right).
\end{align*}
The right-hand side vanishes because $\rho$ is Lagrangian, and it follows that $A(\lambda,t)$ is selfadjoint.

In particular, this implies the graph $G_{\rho}(A(\lambda,t)) \subset \cH$ is Lagrangian, and hence maximal. We also have from the definition of $A$ and \eqref{betatrace} that
\[
	(f,g) + JA(\lambda,t)(f,g) = \tr_t (u_{\lambda,t})
\]
for any $(f,g) \in \rho$, and so $G_{\rho}(A(\lambda,t)) \subset \mu(\lambda,t)$. Since $\mu(\lambda,t)$ is isotropic by Lemma \ref{muisotropic}, the maximality of the graph implies $G_{\rho}(A(\lambda,t)) = \mu(\lambda,t)$. Therefore $\mu(\lambda,t) \subset \cH$ is Lagrangian and the corresponding orthogonal projections in $B(\cH)$ vary smoothly with respect to $\lambda$ and $t$.
\end{proof}

\subsection{The boundary space}
We next discuss the subspace $\nu$ defined in \eqref{boundary}.

\begin{lemma} The boundary space $\nu \subset \cH$ is Lagrangian.
\end{lemma}

\begin{proof} We first observe that $\nu$ can be decomposed as
\begin{align} \label{nusplit}
	\nu = \gamma(\cX) \oplus R \left[ \gamma(\cX)^{\bot} \right],
\end{align}
where $\gamma(\cX)^{\bot}$ denotes the orthogonal complement of $\gamma(\cX)$ in $H^{1/2}(\pO)$. 
%In general, $R ((\gamma (X))^\perp ) = (R (\gamma (X)))^\perp$, because $g \in R ((\gamma (X))^\perp )$ iff $R^{-1}g \in (\gamma (X))^\perp$ iff $\left<R^{-1}g,\gamma u\right>_{H^{1/2}} = 0$ for all $u \in \cX$ iff $\left<g,R \gamma u\right>_{H^{-1/2}} = 0$ for all $u \in \cX$ iff $g \in (R (\gamma (X)))^\perp$. 
By definition, $g \in R \left[ \gamma (X)^\perp \right]$ if and only if $\left<R^{-1}g,\gamma u\right>_{H^{1/2}(\pO)} = 0$ for all $u \in \cX$. Since $\left<R^{-1}g,\gamma u\right>_{H^{1/2}(\pO)} = \left<g,R \gamma u\right>_{H^{-1/2}(\pO)}$, this implies
$R \left[ \gamma (X)^\perp \right] = \left[ R \gamma(\cX) \right]^\perp$.

The subspace $\nu \subset \cH$ is closed because $\gamma\colon H^1(\Omega) \ra H^{1/2}(\pO)$ admits a bounded right inverse, and \eqref{boundary} implies $\nu$ is isotropic. A direct computation shows that
\begin{equation}
\label{eqn:nuperp}
	J\nu = \gamma(\cX)^{\bot} \oplus R \left[ \gamma(\cX) \right] = \nu^{\bot} ,
\end{equation}
hence $\nu$ is Lagrangian.
\end{proof}

The boundary space $\nu$ is rather special within the class of Lagrangian subspaces. It decomposes as a direct sum of $H^{1/2}(\pO)$ and $H^{-1/2}(\pO)$ factors, as in \eqref{nusplit}, so $(f,g) \in \nu$ precisely when both $(f,0)$ and $(0,g)$ are contained in $\nu$. This fact, which is not true for arbitrary Lagrangian subspaces, is a key ingredient in the proof of the following energy estimate, which is essential to the proof of Lemma \ref{lem:Fredholm}.

\begin{lemma} \label{lem:energy}
Let $P_{\nu}$ denote the $\cH$-orthogonal projection onto $\nu$, and $P_{\nu}^{\bot} = I-P_{\nu}$ the projection onto $\nu^{\bot}$.
There is a constant $C = C(\lambda,t)$ such that
\[
	\|u\|_{H^1(\Omega)}^2 \leq C \left( \|u\|_{L^2(\Omega)}^2 + \left\| P_{\nu}^{\bot} \tr_t u \right\|_\cH^2 \right)
\]
for each $u \in K_{\lambda,t}$.
\end{lemma}

\begin{proof} It follows from \eqref{weakGreen} and the coercivity of $D_t$ that there exists $C' > 0$ with
\[
	\|u\|_{H^1(\Omega)}^2 \leq C' \left( \|u\|_{L^2(\Omega)}^2 + \int_{\pO} (B_t u)(\gamma u)d\mu \right)
\]
for $u \in K_{\lambda,t}$. For $\tr_t u = (f,g)$ we define
\begin{align*}
	(f_1,g_1) &= P_{\nu}(f,g) , \\
	(f_2,g_2) &= P_{\nu}^{\bot}(f,g),
\end{align*}
so that $f = f_1 + f_2$ and $g = g_1 + g_2$. We compute
\begin{align*}
	\int_{\pO} (B_t u)(\gamma u)d\mu &= \int_{\pO} (f_1 + f_2)(g_1 + g_2) d\mu \\
	& = \int_{\pO} f_1 g_2 d\mu + \int_{\pO} f_2 g_1 d\mu,
\end{align*}
using the fact that
\[
	\int_{\pO} f_1 g_1 d\mu = \left<f_1, R^{-1} g_1 \right>_{H^{1/2}(\pO)} = 0
\]
by \eqref{nusplit} because $(f_1,g_1) \in \nu$, and similarly for $(f_2,g_2) \in \nu^{\bot}$ using \eqref{eqn:nuperp}. Therefore
\begin{align*}
	\int_{\pO} (B_t u)(\gamma u)d\mu &\leq \epsilon \|(f_1,g_1)\|_{\cH}^2 + (4\epsilon)^{-1} \|(f_2,g_2)\|_{\cH}^2  \\
	&\leq \epsilon C'' \|u\|_{H^1(\Omega)}^2 + (4\epsilon)^{-1}\left\| P_{\nu}^{\bot} \tr_t u \right\|^2
\end{align*}
for any $\epsilon > 0$, and the result follows.%. We choose $\epsilon$ small enough that $\epsilon C' C'' < 1$ and the result follows.
\end{proof}

\subsection{The intersection}
\label{ss:int}
We complete the section by proving the Fredholm property of $\mu(\lambda,t)$ and $\nu$, and giving a proof of Proposition \ref{intersection}.

\begin{lemma} \label{lem:Fredholm}
For each $(\lambda,t) \in \bbR \times [a,b]$, $\mu(\lambda,t)$ and $\nu$ are a Fredholm pair.
\end{lemma}

\begin{proof}
For convenience we fix $(\lambda,t)$ and abbreviate $\mu = \mu(\lambda,t)$. Proposition \ref{intersection} yields $\dim(\mu \cap \nu) = \dim \ker(L_{\cX,t} - \lambda)$, which is finite by Theorem 4.10 of \cite{M00} (cf. Theorem 7.21 of \cite{F95}). Temporarily assuming $\mu+\nu$ is closed, and using that $\mu$ and $\nu$ are Lagrangian, we find that the codimension of $\mu + \nu$ equals the dimension of
\[
	\left( \mu + \nu \right)^{\bot} = \mu^{\bot} \cap \nu^{\bot} = J \mu \cap J \nu = J (\mu \cap \nu),
\]
which is finite because $J$ is an isomorphism.

To prove that $\mu+\nu$ is closed it suffices, by Theorem IV.4.2 of \cite{K76}, to show that the number
\begin{align} \label{gammadef}
	\kappa := \inf_{x \in \mu, x \notin \nu} \frac{\text{dist}(x,\nu)}{\text{dist}(x,\mu \cap \nu)}
\end{align}
is positive. Let $P$ and $\widehat{P}$ denote the orthogonal projections onto $\nu$ and $\mu \cap \nu$, respectively, so that $\text{dist}(x,\nu) = \|x - Px\|_{\cH}$ and $\text{dist}(x,\mu \cap \nu) = \|x - \widehat{P}x\|_{\cH}$.

We first show that there is a positive constant $K$ such that
\begin{align} \label{distbound}
	\|x\|_{\cH} \leq K \|x - Px\|_{\cH}
\end{align}
for all $x \in \mu \cap (\mu \cap \nu)^{\bot}$. Suppose not, so there exists a sequence $\{u_i\}$ in $K_{\lambda,t} \subset H^1(\Omega)$ such that the traces $x_i = \tr_t u_i$ are orthogonal to $\mu \cap \nu$ and satisfy
\[
	\|x_i\|_{\cH} \geq i \|x_i - Px_i\|_{\cH}.
\]
Rescaling, we can assume that $\|u_i\|_{L^2(\Omega)} = 1$ for each $i$. It follows from Lemma \ref{lem:energy} that
\[
	\|u_i\|_{H^1(\Omega)}^2 \leq C \left(1 + i^{-1} \|u_i\|_{H^1(\Omega)}^2\right),
\]
hence the sequence $\{u_i\}$ is bounded in $H^1(\Omega)$, and there exists an element $\bar{u} \in H^1(\Omega)$ and a subsequence $\{u_i\}$ such that $u_i \rightarrow \bar{u}$ in $L^2(\Omega)$ and $u_i \rightharpoonup \bar{u}$ in $H^1(\Omega)$. This implies $\|\bar{u}\|_{L^2(\Omega)} = 1$ and $D_t(u_i,v) \ra D_t(\bar{u},v)$ for any $v \in H^1(\Omega)$, hence $\bar{u} \in K_{\lambda,t}$ and $\tr_t \bar{u} \in \cH$ is well defined. 
Lemma \ref{lemma:trprop} implies $\{x_i\}$ is bounded, so there is a weakly convergent subsequence $x_i \rightharpoonup \bar{x}$ in $\cH$. Since weak limits are unique and $\tr_t u_i \rightharpoonup \tr_t \bar{u}$ (cf. the proof of Lemma \ref{Carleman}), we have that $\bar{x} = \tr_t \bar{u}$. We also have $\|x_i - P x_i\|_{\cH} \ra 0$, hence $\bar{x} \in \nu$. Finally, since each $x_i \in (\mu \cap \nu)^{\bot}$, the weak convergence $x_i \rightharpoonup \bar{x}$ implies $\bar{x} \in (\mu \cap \nu)^{\bot}$ and we conclude that $\bar{x} = 0$. By Lemma \ref{Carleman} this implies $\bar{u} = 0$, a contradiction. This completes the proof of \eqref{distbound}.

Recalling that $P$ and $\widehat{P}$ are the orthogonal projections onto $\nu$ and $\mu \cap \nu$, and letting $x \in \mu$, we thus have
\begin{align*}
	\text{dist}(x,\mu \cap \nu) &= \|x - \widehat{P} x\|_{\cH} \\
	& \leq K \| (x - \widehat{P} x) - P(x - \widehat{P} x)\|_{\cH} \\
	&= K \| x - Px \|_{\cH}
\end{align*}
where in the last equality we have used the fact that $P \widehat{P} = \widehat{P}$ because $\mu \cap \nu \subset \nu$. Referring to \eqref{gammadef}, we have shown that $\kappa \geq K^{-1} > 0$, hence $\mu + \nu$ is closed.

\end{proof}

We conclude with the proof of Proposition \ref{intersection}, first proving a simple lemma about the Dirichlet trace restricted to a subspace of $H^1(\Omega)$.  

\begin{lemma} \label{traceX} Let $\cX \subset H^1(\Omega)$ be a subspace that contains $H^1_0(\Omega)$, and suppose $u \in H^1(\Omega)$. Then $\gamma u \in \gamma(\cX)$ if and only if $u \in \cX$.
\end{lemma}

\begin{proof} Suppose $\gamma u \in \gamma(\cX)$, so there exists $w \in \cX$ with $\gamma u = \gamma w$, hence $\gamma(u-w) = 0$. This implies $u - w \in H^1_0(\Omega) \subset \cX$, so $u = (u-w) + w \in \cX$.
\end{proof}

\begin{proof}[Proof of Proposition \ref{intersection}]
First suppose there exists a nonzero function $u \in \cD(L_{\cX,t})$ with $L_{\cX,t} u = \lambda u$. Then $D_t(u,v) = \lambda \left<u,v\right>_{L^2(\Omega)}$ for all $v \in \cX$, hence for all $v \in H^1_0(\Omega)$, and so $u \in K_{\lambda,t}$. From \eqref{weakGreen} we find
\[
	\int_{\pO} ( B_t u)(\gamma v) d\mu = 0
\]
for all $v \in \cX$, which implies $\tr_t u = (\gamma u, B_t u) \in \mu(\lambda,t) \cap \nu$. It follows from Lemma \ref{Carleman} that $\mu(\lambda,t) \cap \nu \neq \{0\}$.

Now suppose that $\mu(\lambda,t) \cap \nu \neq \{0\}$. By definition, there exists $u \in K_{\lambda,t}$ with nonvanishing trace $\tr_t u \in \mu(\lambda,t) \cap \nu$. Since $\tr_t u \in \nu$ we have $\gamma u \in \gamma(\cX)$, hence $u \in \cX$ by Lemma \ref{traceX}. We also have from the definition of $\nu$ that
\[
	\int_{\pO} ( B_t u) (\gamma v) d\mu = 0,
\]
 and hence
\[
	D_t(u,v) = \left<L_t u,v\right>_{L^2(\Omega)},
\]
for all $v \in \cX$. It follows that $u \in \cD(L_{\cX,t})$ and $L_{\cX,t} u = L_t u = \lambda u$.
\end{proof}

%%%%%%%%%%%%%%%%%%%%%%%%%%%%%%%%%%%%%%%%
%%%%%%%%%%%%%%%%%%%%%%%%%%%%%%%%%%%%%%%%
%%%%%%%%%%%%%%%%%%%%%%%%%%%%%%%%%%%%%%%%
%%%%%%%%%%%%%%%%%%%%%%%%%%%%%%%%%%%%%%%%
%%%%%%%%%%%%%%%%%%%%%%%%%%%%%%%%%%%%%%%%

\section{Proof of Theorem \ref{thm:main}}
\label{sec:proof}
We now prove the main theorem. As in \cite{DJ11}, this follows from the homotopy invariance of the Maslov index, along with a monotonicity computation and a uniform lower bound on the eigenvalues of $L_{\cX,t}$.

For any fixed $\lambda_0 < 0$, $\mu(\lambda,t)$ defines a homotopy $[\lambda_0,0] \times [a,b] \rightarrow \mathcal{F} \Lambda_{\nu}(\mathcal{H})$, hence
\begin{align}
	\Mas(\mu(\lambda,a); \nu) + \Mas(\mu(0,t); \nu) = \Mas(\mu(\lambda_0,t); \nu) + \Mas(\mu(\lambda,b); \nu).
\label{eqn:htpy}
\end{align}
To prove Theorem \ref{thm:main} we analyze each term in the above equation.

\begin{lemma} There exists a constant $\lambda_0 < 0$ such that $\mu(\lambda,t) \cap \nu = \{0\}$ for all $t \in [a,b]$ and $\lambda \leq \lambda_0$.
\label{lemma:semibound}
\end{lemma}

In other words, the operators $L_{\cX,t}$ have eigenvalues bounded uniformly below for $t \in [a,b]$, so we can choose $\lambda_0$ to ensure $\Mas(\mu(\lambda_0,t); \nu) = 0$.

\begin{proof} By Proposition \ref{intersection} it suffices to show that $D_t(u,u) \geq C \|u\|_{L^2(\Omega)}$ for all $u \in H^1(\Omega)$ and $t \in [a,b]$, where $C \in \bbR$ is independent of $t$. This follows from the continuity of the coefficients of $D_t$ with respect to $t$ and the compactness of the interval $[a,b]$ (cf. the proof of Proposition \ref{Dform} in Appendix \ref{forms}).
\end{proof}

The following lemma, along with \eqref{eqn:htpy}, completes the proof of Theorem \ref{thm:main}.

\begin{lemma} If $t_0 \in [a,b]$, then $\Mas(\mu(\lambda,t_0); \nu) = -\Mor(L_{\cX,t_0})$.
\label{lemma:morse}
\end{lemma}

\begin{proof} Since the path $\lambda \mapsto \mu(\lambda,t_0)$ is smooth, we can determine its Maslov index using crossing forms. We claim that the path is negative definite (as defined in Appendix \ref{app:Maslov}) hence
\begin{align*}
	\Mas(\mu(\lambda,t_0); \nu) &= -\sum_{\lambda_0 \leq \lambda < 0} \dim \left[ \mu(\lambda,t_0) \cap \nu \right] \\
	&= -\sum_{\lambda < 0} \dim \left[ \mu(\lambda,t_0) \cap \nu \right] \\
	&= - \Mor(L_{\cX,t_0}),
\end{align*}
where in the last two equalities we have used Lemma \ref{lemma:semibound} and Proposition \ref{intersection}, respectively.

To prove the claimed monotonicity, we assume there is a crossing at $\lambda_*$, so there exists a path $\{x_{\lambda}\}$ in $\cH$ with $x_{\lambda} \in \mu(\lambda,t_0)$ for $|\lambda - \lambda_*| \ll 1$ and $x_{\lambda_*} \in \nu$. By Lemma \ref{Carleman} there is a path $\{u_{\lambda}\}$ in $H^1(\Omega)$ such that $\tr_{t_0} u_{\lambda} = x_{\lambda}$. Differentiating the equation $D_{t_0}(u_{\lambda},v) = \lambda \left<u_{\lambda},v\right>_{L^2(\Omega)}$ with respect to $\lambda$ and letting $' = \frac{d}{d\lambda}$, we find
\[
	D_{t_0}(u'_{\lambda},v) = \left<\lambda u'_{\lambda} + u_{\lambda},v\right>_{L^2(\Omega)}
\]
for all $v \in H^1_0(\Omega)$, so \eqref{weakGreen} implies
\begin{align*}
	D_{t_0}(u'_{\lambda},u_{\lambda}) &= \left<\lambda u'_{\lambda} + u_{\lambda},u_{\lambda} \right>_{L^2(\Omega)} + \int_{\pO} \left( B_{t_0} u'_{\lambda} \right)u_{\lambda} d\mu,  \\
	D_{t_0}(u_{\lambda},u'_{\lambda}) &= \left<\lambda u_{\lambda},u'_{\lambda} \right>_{L^2(\Omega)} + \int_{\pO} \left( B_{t_0} u_{\lambda} \right) u'_{\lambda} d\mu .
\end{align*}
Since $D_{t_0}$ is symmetric, we obtain
\begin{align*}
	Q(x_{\lambda_*},x_{\lambda_*}) &= \left. \omega \left( \tr_{t_0} u_{\lambda}, \tr_{t_0} u'_{\lambda} \right) \right|_{\lambda=\lambda_*} \\
	& = - \|u_{\lambda_*}\|_{L^2(\Omega)}^2,
\end{align*}
which is negative because $u_{\lambda_*}$ is not identically zero.
\end{proof}

%%%%%%%%%%%%%%%%%%%%%%%%%%%%%%%%%%%%%%%%
%%%%%%%%%%%%%%%%%%%%%%%%%%%%%%%%%%%%%%%%
%%%%%%%%%%%%%%%%%%%%%%%%%%%%%%%%%%%%%%%%
%%%%%%%%%%%%%%%%%%%%%%%%%%%%%%%%%%%%%%%%
%%%%%%%%%%%%%%%%%%%%%%%%%%%%%%%%%%%%%%%%

\section{The crossing form}
\label{sec:mono}

Having completed the proof of Theorem \ref{thm:main}, we study the Maslov index on the left-hand side of \eqref{eqn:MM} in greater detail. This is a signed count of the conjugate times in $[a,b]$, with the sign depending on the direction in which the subspace $\mu(0,t)$ passes though $\nu$. This is intimately related to the motion of the eigenvalues of $L_{\cX,t}$ with respect to $t$, which depends nontrivially on the boundary conditions. We elucidate this dependence by computing crossing forms for the Dirichlet and Robin problems introduced in Section \ref{sec:results}, corresponding to the spaces $\cX^0 = H^1_0(\Omega)$ and $\cX^1 = H^1(\Omega)$.

We assume throughout that each $\Omega_t$ has $C^{1,1}$ boundary and the coefficients of $D$ are continuously differentiable on $\overline{\cup_{a \leq t \leq b} \Omega_t}$. (This is the true under the hypotheses of either Corollary \ref{cor:Dir} or \ref{cor:Neumannpotential}.) By Lemma \ref{lemmaDtreg} the coefficients of $D_t$ are contained in $C^1([a,b],L^\infty(\Omega))$, and Theorem 4.18 of \cite{M00} implies that if $u \in \ker L_{\cX,t}$, then $u \circ \varphi_t^{-1} \in H^2(\Omega_t)$.

\subsection{The general framework}

We start with some computations that are valid for any boundary conditions, letting $D_t'$ denote the derivative of the form $D_t$ with respect to $t$, so that
\[
	\frac{d}{dt} D_t(u_t,v_t) = D_t'(u_t,v_t) + D_t(u_t',v_t) + D_t(u_t,v_t')
\]
when $u_t,v_t$ are differentiable paths in $H^1(\Omega)$.

\begin{lemma} Suppose $U \subset [a,b]$ is open and $u_t \in C^1(U, H^1(\Omega))$. If $u_t \in K_{0,t}$ for each $t \in U$, then
\begin{align} \label{crossing1}
	\omega \left( \tr_t u_t, (\tr_t u_t)'  \right) = D_t'(u_t,u_t),
\end{align}
where $' = d/dt$.
\end{lemma}

\begin{proof}
From the definition of $\omega$ we have
\begin{align*}
	\omega \left( \tr_t u_t, (\tr_t u_t)'  \right) &= \int_{\pO} \left[ (B_t u_t)' \gamma u_t - (B_t u_t)\gamma u_t' \right] d\mu.
\end{align*}
Recalling that $D_t(u_t,v) = \int_{\pO} (B_t u_t) (\gamma v) d \mu$ for all $v \in H^1(\Omega)$, we differentiate with respect to $t$ and then evaluate at $v = u_t$ to find
\[
	D_t'(u_t,u_t) + D_t(u_t',u_t) = \int_{\pO} (B_t u_t)'(\gamma u_t) d\mu.
\]
We also have
\[
	D_t(u_t,u_t') = \int_{\pO} (B_t u_t)(\gamma u_t') d\mu
\]
and the result follows from the symmetry of $D_t$.

\end{proof}

It thus remains to compute $D_t'(u_t,u_t)$ when $t$ is a conjugate time. We start by writing the Dirichlet form $D$ abstractly as
\begin{align} \label{Dabs}
	D(u,u) = \int_{\Omega} F(u,\nabla u).
\end{align}

\begin{prop} \label{prop:crossing}
Suppose $t_* \in [a,b]$ is a conjugate time, with $u_{t_*} \in \ker L_{\cX,t_*}$. Let $\widehat{u} = u_{t_*} \circ \varphi_{t_*}^{-1}$ and $x_* = \tr_{t_*} u_{t_*}$. Then the crossing form satisfies
\begin{align} \label{crossing}
	Q(x_*,x_*) = \int_{\pOt} \left[F \left(\widehat{u},\nabla \widehat{u} \right)(X \cdot N_t)  - 2 (\widehat{B}_t \widehat{u})(X\widehat{u}) \right]d\mu_t
\end{align}
where $X = \varphi_t'$, $N_t$ is the outward unit normal to $\pOt$, $d\mu_t$ is the induced volume form on $\partial \Omega_t$ and $\widehat{B}_t $ is the boundary operator defined in \eqref{Bhatdef}.
\end{prop}

\begin{proof} From \eqref{Dabs} and the definition of $D_t$ we have
\begin{align*}
	D_t(u,u) &= \left. D\right|_{\Omega_t} (u \circ \varphi_t^{-1}, u \circ \varphi_t^{-1}) \\
	&= \int_{\Omega_t} F \left(u \circ \varphi_t^{-1},\nabla (u\circ \varphi_t^{-1}) \right).
\end{align*}
Differentiating and using Theorem 1.11 from \cite{H05} we obtain
\begin{align*}
	D_t'(u,u) =& -2 \left. D\right|_{\Omega_t} \left(X(u \circ \varphi_t^{-1}), u \circ \varphi_t^{-1} \right) \\
	&+ \int_{\pOt} F \left(u \circ \varphi_t^{-1},\nabla (u\circ \varphi_t^{-1}) \right) (X \cdot N_t) d\mu_t.
\end{align*}
Setting $t=t_*$ and $u = u_{t_*}$, the result follows.
\end{proof}

We now consider some specific examples for the operator $L = -\Delta + V(x)$.

\subsection{The Dirichlet ($\cX^0$) problem}
We use the Dirichlet form
\begin{align} \label{generalD}
	D(u,v) = \int_{\Omega} \left[ \nabla u \cdot \nabla v + Vuv \right],
\end{align}
which has rescaled boundary operator $\widehat{B}_t  =  \partial/\partial N_t$. Suppose that $t_*$ is a crossing time. With $x_*$, $u_{t_*}$ and $\widehat{u}$ as in Proposition \ref{prop:crossing} we have
\begin{align} \label{gencrossing}
	Q(x_*,x_*) = \int_{\pOt} \left[ \left( |\nabla \widehat{u}|^2 + V \widehat{u}^2 \right)(X \cdot N_t) - 2 X \widehat{u} \frac{\partial \widehat{u}}{\partial N_t} \right] d\mu_t.
\end{align}
Since $\widehat{u}$ vanishes on $\pOt$, this reduces to
\[
	Q(x_*,x_*) = \int_{\pOt} \frac{\partial \widehat{u}}{\partial N_t} \left[ (X \cdot N_t) \frac{\partial \widehat{u}}{\partial N_t}- 2 X \widehat{u} \right] d\mu_t.
\]
We decompose the velocity field $X$ into normal and tangential components, $X = X^{\top} + (X \cdot N_t) N_t$, and observe that
\begin{align*}
	X \widehat{u} = (X \cdot N_t) \frac{\partial \widehat{u}}{\partial N_t}
\end{align*}
because $X^{\top} \widehat{u} = 0$. It follows that
\begin{align} \label{Dcrossing}
	Q(x_*,x_*) = - \int_{\pOt} \left(\frac{\partial \widehat{u}}{\partial N_t} \right)^2 (X \cdot N_t) d\mu_t.
\end{align}

More generally, for the operator $L = -\partial_i (a^{ij} \partial_j) + c$, the same computation yields
\[
	Q(x_*,x_*) = - \int_{\pOt} a(N_t, N_t) \left(\frac{\partial \widehat{u}}{\partial N_t} \right)^2 (X \cdot N_t) d\mu_t,
\]
where $a(\cdot,\cdot)$ denotes the bilinear form corresponding to $a^{ij}$. In either case, we see that crossings for the Dirichlet problem are isolated and negative definite as long as $X \cdot N_t > 0$; the proof of Corollary \ref{cor:Dir} follows. Geometrically the condition $X \cdot N_t > 0$ means that $\Omega_t$ is moving outward as $t$ increases. If $X \cdot N_t$ changes sign on $\pOt$, then the signature of the crossing form is more difficult to determine, as it depends on the structure of $\partial \widehat{u} / \partial N_t$ on the boundary.

The expression for $Q$ given in \eqref{gencrossing} is valid for any boundary value problem corresponding to the Dirichlet form \eqref{generalD} (that is, for any choice of $\cX$). In particular, we can use this to compute crossing forms for the Neumann problem, as well as the $\cX^2$ and $\cX^3$ problems formulated above. The Robin boundary value problem requires a modification to the form and is considered in detail below.

\subsection{The Robin ($\cX^1$) problem}
We now consider the Dirichlet form
\[
	D(u,v) = \int_{\Omega} \left[ \nabla u \cdot \nabla v + \dv(uvY) + Vuv \right]
\]
where $Y$ is vector field of class $C^2$ (hence $\dv Y$ is $C^1$). This corresponds to $L = -\Delta + V$, with the boundary operator
\[
	\widehat{B}_t u = \frac{\partial u}{\partial N_t} + \beta_t u
\]
on $\Omega_t$, where we have defined $\beta_t = (Y|_{\pOt}) \cdot N_t$. Without loss of generality we may assume that $Y$ has no component tangential to $\pO_{t_*}$, hence $Y|_{\pO_{t_*}} = \beta_{t_*} N_{t_*}$.

Since $\cX^1 = H^1(\Omega)$ we have $\widehat{B}_t \widehat{u} = 0$ at a crossing time, so Proposition \ref{prop:crossing} yields
\[
	Q(x_*,x_*) = \int_{\pOt} \left[ |\nabla \widehat{u}|^2 + \dv (\widehat{u}^2 Y) +  V(y) \widehat{u}^2 \right] (X \cdot N_t) d\mu_t.
\]
Using the fact that $Y|_{\pOt} =\beta_t N_t$ to compute the second term explicitly, we obtain
\begin{align*}
	\dv (\widehat{u}^2 Y) &= \dv \left( \beta_t \widehat{u}^2 N_t \right) \\
	&= \frac{\partial \beta_t}{\partial N_t} \widehat{u}^2 + 2 \beta_t \widehat{u} \frac{\partial \widehat{u}}{\partial N_t} + \beta_t \widehat{u}^2 \dv N_t \\
	&= \left( \frac{\partial \beta_t}{\partial N_t} - 2 \beta_t^2 + \beta_t H_{\pOt} \right) \widehat{u}^2
\end{align*}
where we have used the fact that $\widehat{B}_t \widehat{u} = 0$, and the mean curvature is defined to be $H_{\pOt} = \dv N_t$. Decomposing $\nabla \widehat{u} = \nabla^{\top} \widehat{u} + \frac{\partial \widehat{u}}{\partial N_t}$ into tangential and normal components, then applying the boundary conditions, we have
\[
	|\nabla \widehat{u}|^2 = |\nabla^{\top} \widehat{u}|^2 + \beta_t^2 \widehat{u}^2,
\]
and so
\begin{align} \label{Rcrossing}
	Q(x_*,x_*) =  \int_{\pOt} \left[ |\nabla^{\top} \widehat{u}|^2+ \left( V - \beta_t^2 + \beta_t H_{\pOt} +  \frac{\partial \beta_t}{\partial N_t} \right) \widehat{u}^2 \right] (X \cdot N_t) d\mu_t.
\end{align}

This crossing form coincides with the formula for the first variation of a simple Robin eigenvalue in equation (4.3) of \cite{BW2014} and example 3.5 of \cite{H05}. (The computations in the proof of the latter reference agree perfectly with ours, but the final result on p. 40 contains an extra factor of $2$ on the $\p\beta_t / \p N_t$ term.) One advantage of the symplectic formulation is that it describes the change in the Morse index, rather than the individual eigenvalues, and hence is robust against multiplicities and degeneracies.

\subsection{The star-shaped case}

We finally revisit the star-shaped case for the Dirichlet and Neumann problems. With
\[
	D(u,v) = \int_{\Omega} \left[ \nabla u \cdot \nabla v + V(x)uv \right] dx
\]
as above, $\varphi_t(x) = tx$ and $\Omega_t = \{tx : x \in \Omega\}$, a simple computation shows that
\[
	D_t(u,v) = t^{n-2} \int_{\Omega} \left[ (\nabla u \cdot \nabla v) + t^2 V(tx) uv \right] dx
\]
and so
\begin{align*}
	D_t'(u,v) =& (n-2) t^{n-3} \int_{\Omega} \left[ (\nabla u \cdot \nabla v) + t^2 V(tx) uv \right] dx \\
	&+ t^{n-2} \int_{\Omega} \frac{d}{dt} \left[ t^2 V(tx) \right] uv \ dx.
\end{align*}
Evaluating at a solution $u_t$ to the equation $-\Delta u_t + t^2 V(tx) u_t = 0$ (i.e. $L_t u_t = 0$) we find that
\begin{align}
	D_t'(u_t,u_t) = (n-2) t^{n-3} \int_{\pO} u_t \frac{\partial u_t}{\partial N} d\mu
	+ t^{n-2} \int_{\Omega} u_t^2 \frac{d}{dt} \left[ t^2 V(tx) \right] dx.
\end{align}

In particular, for either Dirichlet or Neumann boundary conditions, we have
\begin{align}
\label{eqn:ss-mono1}
	D_t'(u_t,u_t) = t^{n-2} \int_{\Omega} u_t^2 \frac{d}{dt} \left[ t^2 V(tx) \right] dx.
\end{align}
Replacing $V(x)$ with $V(x) - \lambda$, this becomes
\begin{align}
\label{eqn:ss-mono2}
	D_t'(u_t,u_t) = t^{n-2} \int_{\Omega} u_t^2 \frac{d}{dt} \left[ t^2 V(tx) - t^2 \lambda \right] dx,
\end{align}
and we conclude that all crossings are negative definite provided
\[
	\frac{d}{dt} \left[ t^2 V(tx) - t^2 \lambda \right] < 0
\]
for all $x \in \Omega$. This is equivalent to \eqref{biglambda}, so Corollary \ref{cor:Neumannpotential} follows immediately.

%%%%%%%%%%%%%%%%%%%%%%%%%%%%%%%%%%%%%%%%
%%%%%%%%%%%%%%%%%%%%%%%%%%%%%%%%%%%%%%%%
%%%%%%%%%%%%%%%%%%%%%%%%%%%%%%%%%%%%%%%%
%%%%%%%%%%%%%%%%%%%%%%%%%%%%%%%%%%%%%%%%
%%%%%%%%%%%%%%%%%%%%%%%%%%%%%%%%%%%%%%%%

\appendix

\section{Selfadjoint operators and bilinear forms}
\label{forms}
In this appendix we review the correspondence between symmetric bilinear forms and selfadjoint operators described in Proposition \ref{elliptic}. While the result is standard, it is worth reviewing, as the constructions in the paper (in particular for the boundary space) rely on an explicit identification of the domain of the unbounded operator corresponding to a given form.

Our starting point is a symmetric bilinear form
\[
	D(u,v) = \int_{\Omega} \left[ a^{ij}(\partial_i u)(\partial_j v) + b^i (\partial_i u)v + b^i u (\partial_i v) + c uv \right]
\]
with real coefficients $a^{ij}, b^i, c \in L^{\infty}(\Omega)$ satisfying $a^{ij} = a^{ji}$. We assume that $D$ is strongly elliptic, so there exists a constant $\lambda_0 > 0$ such that
\[
	a^{ij}(x) \xi_i \xi_j \geq \lambda_0 | \xi |^2
\]
for all $x \in \Omega$ and $\xi \in \bbR^n$. Formally integrating by parts, we find
\begin{align}
	D(u,v) = \int_{\Omega} \left[ -\partial_i(a^{ij} \partial_j u)  + (c - \partial_i b^i) u\right] v 
	+ \int_{\pO} N_i \left( a^{ij} \partial_j u + b^i u\right) v\,d\mu
\end{align}
where $\{N_j\}$ are the components of the outward-pointing unit normal to $\pO$. The following weak version of Green's formula (Theorem 4.4 of \cite{M00}) justifies this computation.

\begin{lemma}\label{APPgreen}
Let $L = -\partial_i (a^{ij} \partial_j) + (c - \partial_i b^i)$ and $B = N_i (a^{ij} \partial_j + b^i)$. If $u,v \in H^1(\Omega)$ and $Lu \in L^2(\Omega)$, then
\[
	D(u,v) = \left<Lu,v\right>_{L^2(\Omega)} + \int_{\pO} (Bu)(\gamma v) d\mu.
\]
\end{lemma}

%It is easily shown that such a $D$ is coercive and bounded over $H^1(\Omega)$, i.e. there exist constants $C_1,C_2 > 0$ and $\lambda \geq 0$ such that
%\begin{align} \label{coercivity}
%	D(u,u) \geq C_1 \|u\|^2_{H^1(\Omega)} - \lambda \|u\|^2_{L^2(\Omega)}
%\end{align}
%and
%\[
%	|D(u,v)| \leq C_2 \|u\|_{H^1(\Omega)} \|v\|_{H^1(\Omega)}
%\]
%for all $u,v \in H^1(\Omega)$.

We let $\cX$ be a closed subspace of $H^1(\Omega)$ that contains $H^1_0(\Omega)$, and view $D$ as an unbounded form on $L^2(\Omega)$ with domain $\cX$. Since $D$ is bounded and coercive over $H^1(\Omega)$, and hence over $\cX$, Theorem VIII.15 of \cite{RS72} (cf. Theorem VI.2.1 of \cite{K76}) implies there is a selfadjoint operator $L_{\cX}$, with domain
\[
	\cD(L_\cX) = \left\{u \in \cX : \exists w \in L^2(\Omega) \text{ with } \left<w,v\right>_{L^2(\Omega)} = D(u,v) \text{ for all } v \in \cX \right\},
\]
satisfying
\[
	D(u,v) = \left<L_{\cX}u,v\right>_{L^2(\Omega)}
\]
for all $u \in \cD(L_{\cX})$ and $v \in \cX$. We can identify the domain explicitly in terms of the operators $L$ and $B$.

\begin{lemma} Let $u \in \cX$. Then $u \in \cD(L_{\cX})$ if and only if $Lu \in L^2(\Omega)$ and
\[
	\int_{\pO} (Bu)(\gamma v) d\mu = 0
\]
for every $v \in \cX$.
\end{lemma}

\begin{proof}
Suppose $u \in \cX$, with $Lu \in L^2(\Omega)$ and $\int_{\pO} (Bu)(\gamma v) d\mu = 0$
for all $v \in \cX$. Then Lemma \ref{APPgreen} implies $D(u,v) = \left<Lu,v\right>_{L^2(\Omega)}$ for all $v \in \cX$, hence $u \in \cD(L_\cX)$ and $L_\cX u = Lu$.

On the other hand, suppose $u \in \cD(L_\cX)$. Since $D(u,v) = \left<L_\cX u,v\right>_{L^2(\Omega)}$ for all $v \in H^1_0(\Omega)$, we have $Lu = L_\cX u \in L^2(\Omega)$. It follows from Lemma \ref{APPgreen} that
\[
	 \left<L_\cX u,v\right>_{L^2(\Omega)} = D(u,v) = \left<Lu, v\right>_{L^2(\Omega)} + \int_{\pO} (Bu)(\gamma v) d\mu 
\]
for all $v \in \cX$, hence
\[
	\int_{\pO} (Bu)(\gamma v) d\mu = 0
\]
as claimed.
\end{proof}

%%%%%%%%%%%%%%%%%%%%%%%%%%%%%%%%%%%%%%%%
%%%%%%%%%%%%%%%%%%%%%%%%%%%%%%%%%%%%%%%%
%%%%%%%%%%%%%%%%%%%%%%%%%%%%%%%%%%%%%%%%
%%%%%%%%%%%%%%%%%%%%%%%%%%%%%%%%%%%%%%%%
%%%%%%%%%%%%%%%%%%%%%%%%%%%%%%%%%%%%%%%%

\section{The Maslov index in symplectic Hilbert spaces}
\label{app:Maslov}
We next review the definitions and basic properties of symplectic Hilbert spaces, the Fredholm--Lagrangian Grassmannian, and the Maslov index. These will be our main tools in the proof of Theorem \ref{thm:main}. Unless stated otherwise, technical details can be found in \cite{F04}.

Let $H$ be a real, infinite-dimensional, separable Hilbert space, and $\omega\colon H \times H \rightarrow \mathbb{R}$ a bilinear, skew-symmetric form. If the map $x \mapsto \omega(x, \cdot)$ is an isomorphism $H \rightarrow H^*$ we say that $\omega$ is \textit{nondegenerate}, and call the pair $(H,\omega)$ a \textit{symplectic Hilbert space}. For example, if $E$ is a Hilbert space, we can set $H = E \oplus E^*$ and define
\begin{align*}
	\omega((x,\phi), (y,\psi)) = \psi(x) - \phi(y),
\end{align*}
which is easily seen to be nondegenerate.

A subspace $\mu \subset H$ is said to be \textit{isotropic} if $\omega(x,y) = 0$ for all $x,y \in \mu$. A \textit{Lagrangian} subspace is then defined to be a maximal closed, isotropic subspace of $H$. The set of all Lagrangian subspaces in $H$ is called the \textit{Lagrangian Grassmannian} and denoted by $\Lambda(H)$. Given the gap topology (where the distance between subspaces $\mu$ and $\nu$, with respective orthogonal projections $P_{\mu}$ and $P_{\nu}$, is the operator norm $\|P_{\mu} - P_{\nu} \|$), the Lagrangian Grassmannian becomes a smooth, contractible Banach manifold, locally equivalently to the space of bounded, selfadjoint operators on $H$.

If follows that any homotopy invariant $C^0(\mathbb{S}^1; \Lambda(H)) \rightarrow \mathbb{Z}$ is necessarily trivial. This differs from the finite-dimensional case, where we have $\pi_1(\Lambda(\mathbb{R}^{2n})) = \mathbb{Z}$. For this reason we must work with the Fredholm--Lagrangian Grassmannian, which is topologically nontrivial.

%Before doing so, we give a useful criteria for finding Lagrangian subspaces.
%
%\begin{lemma} Suppose $\mu \subset H$ is a closed isotropic subspace. If there exists an isotropic subspace $V \subset H$ such that $H = \mu + V$, then $\mu$ is Lagrangian.
%\label{lemma:lag}
%\end{lemma}
%
%\begin{proof} Since $\mu$ is closed and isotropic, it suffices to prove that it is maximal. Suppose this is not the case, so there exists an isotropic subspace $\nu \supset \mu$ and a nonzero element $z \in \nu \cap V$. Now for $x+y \in \mu + V$ we have
%\begin{align*}
%	\omega(x+y,z) = \omega(x,z) + \omega(y,z).
%\end{align*}
%The first term vanishes since $x,z \in \nu$, and the second term vanishes since $y,z \in V$. But $x+y \in H$ was arbitrary so this contradicts the nondegeneracy of $\omega$ on $H$.
%\end{proof}

We first introduce the notion of a \textit{Fredholm pair} in the Lagrangian Grassmannian. This is a pair of closed subspaces $\mu, \nu \in \Lambda(H)$ such that
\begin{enumerate}
	\item $\dim  (\mu \cap \nu) < \infty$, and
	\item $\mu + \nu$ is closed and of finite codimension in $H$.
\end{enumerate}
Then the \textit{Fredholm--Lagrangian Grassmannian} of $H$, with respect to fixed $\nu \in \Lambda(H)$, is the set
\begin{align*}
	\mathcal{F} \Lambda_{\nu} (H) = \{ \mu \in \Lambda(H) : (\mu, \nu) \textrm{ is a Fredholm pair} \}.
\end{align*}
This is an open subset of $\Lambda(H)$, and hence a smooth Banach manifold, with $\pi_1 (\mathcal{F} \Lambda_{\nu} (H)) \cong \mathbb{Z}$.

We conclude our review by defining the \textit{Maslov index} of a continuous path $\mu\colon [a,b] \rightarrow \mathcal{F} \Lambda_{\nu} (H)$. We define a continuous family of operators $\{S(t)\}$ by the formula
\begin{align*}
	S(t) = (2 P_{\mu(t)} - I) (2 P_{\nu} - I), 
\end{align*}
where $P$ denotes orthogonal projection onto the designated subspace. The operator $S(t)$ comprises reflection across the subspace $\nu$ followed by reflection across $\mu(t)$. There exist times $a = t_0  < t_1 < \cdots < t_N = b$ and positive constants $\epsilon_j \in (0,\pi)$ for $1 \leq j \leq N$ such that
\begin{enumerate}
	\item $e^{ \pm i \epsilon_j} \notin \sigma(S(t))$, and
	\item $\sum_{|\theta| \leq \epsilon_j} \dim \ker \left(S(t) - e^{i \theta} \right) < \infty$
\end{enumerate}
for all $t \in [t_{j-1}, t_j]$. Intuitively this means that as $t$ ranges from $t_{j-1}$ to $t_j$, the number of eigenvalues of $S(t)$ in the arc $\{ |\theta| \leq \epsilon_j \} \subset \mathbb{S}^1$ is constant and finite.

Following \cite{BF98} and \cite{P96}, we define the Maslov index by the formula
\[
	 \Mas(\mu(t); \nu) =  \sum_{j=1}^N \sum_{0 \leq \theta \leq \epsilon_j} \left[ \dim \ker \left(S(t_j) - e^{i \theta} \right) - \dim \ker \left(S(t_{j-1}) - e^{i \theta} \right) \right].
\]
This gives a count (with sign and multiplicity) of the eigenvalues of $S(t)$ that pass through the point $1 \in \mathbb{S}^1$ in a counterclockwise direction as $t$ increases from $a$ to $b$.

To compute the Maslov index in practice, we make frequent use of crossing forms---see \cite{RS93} for the general theory and \cite{RS95} for an application to first-order, elliptic operators. Suppose $\mu \colon [a,b] \ra \cF \Lambda_{\nu}(\cH)$ is a $C^1$ path and $t_* \in [a,b]$ is a \emph{crossing time}, so $\mu(t_*) \cap \nu \neq \{0\}$. For each $t$ close to $t_*$ there exists a bounded operator $A_t \colon \mu(t_*) \ra \mu(t_*)$ such that $\mu(t)$ is the graph
\[
	\mu(t) = G_{\mu(t_*)}(A_t) = \{x + JA_t(x) : x \in \mu(t_*)\}.	
\]
The \emph{crossing form} is the symmetric, bilinear form defined by
\begin{align}
	Q(x,y) = \left. \frac{d}{dt} \omega \left(x, JA_t(y) \right)\right|_{t=t_*}
\end{align}
for all $x$ and $y$ in the finite-dimensional space $\mu(t_*) \cap \nu$. This is useful for the following reason.

\begin{prop} Let $\mu \in C^1([a,b], \cF \Lambda_{\nu}(\cH))$ and suppose $t_* \in [a,b]$ is a crossing time. Assume $Q$ is nondegenerate, with $p$ positive and $q$ negative eigenvalues. If $t_* \in (a,b)$ and $\delta > 0$ is sufficiently small, then
\[
	\Mas \left(\mu(t)|_{[t_*-\delta,t_*+\delta]}; \nu \right) = p-q.
\]
If $t_* = a$, then
\[
	\Mas \left(\mu(t)|_{[a,a+\delta]}; \nu \right) = -q,
\]
and if $t_* = b$, then
\[
	\Mas \left(\mu(t)|_{[b-\delta,b]}; \nu \right) = p.
\]
\end{prop}

In other words, the local contribution to $\Mas(\mu(t);\nu)$ at $t_*$ is determined by the signature of the crossing form. When the crossing occurs at an endpoint of the curve, an initial crossing ($t_*=a$) can only contribute negatively to the Maslov index, while a terminal crossing ($t_*=b)$ can only contribute positively. If $\mu(t)$ is a negative path, in the sense that $Q$ is strictly negative at any crossing time, then
\[
	\Mas(\mu(t); \nu) = -\sum_{t\in[a,b)} \dim \left( \mu(t) \cap \nu \right),
\]
while for a positive curve one has
\[
	\Mas(\mu(t); \nu) = \sum_{t\in (a,b]} \dim \left( \mu(t) \cap \nu \right).
\]

\section{Smooth families of Dirichlet forms} \label{app:regular}
The form $D$ on $\Omega$ is said to be \emph{invertible} if, for any nonzero $u \in H^1(\Omega)$, there exists $v \in H^1(\Omega)$ with $D(u,v) \neq 0$.

\begin{prop} \label{Dform} 
Let
\[
	D_t(u,v) = \int_{\Omega} \left[ a^{ij}_t(\partial_i u)(\partial_j v) + b^i_t (\partial_i u)v + c^i_t u (\partial_i v) + d_t uv \right]
\]
be a one-parameter family of invertible, strongly elliptic Dirichlet forms, defined for $t$ in a compact interval $I$, with coefficients $a^{ij}_t, b^i_t, c^i_t,d_t \in C^k\left(I, L^{\infty}(\Omega) \right)$ for some $k \geq 0$.

Let $\{F_t\}$ be a one-parameter family of bounded linear functionals on $H^1(\Omega)$, contained in $C^k(I; H^1(\Omega)^*)$. Then for each $t \in I$, there exists a unique $u_t \in H^1(\Omega)$ such that $D_t(u_t,v) = F_t(v)$ for every $v \in H^1(\Omega)$. Moreover, the path $t \mapsto u_t$ is contained in $C^k(I, H^1(\Omega))$.
\end{prop}

Therefore, when the boundary-value problem for each $D_t$ is uniquely solvable, the path of solutions $u_t$ will be at least as smooth as the coefficients of $D_t$ and the inhomogeneous term $F_t$.

\begin{proof} From Theorem 4.7 of \cite{M00} (cf. Theorem 7.13 in \cite{F95}) we have that each $D_t$ is coercive, i.e. there exist constants $C_1 > 0$ and $C_2 \in \bbR$ such that
\begin{equation}
	\label{unifcoerc}
	|D_t(u,u)| \geq C_1 \|u\|^2_{H^1(\Omega)} - C_2 \|u\|^2_{L^2(\Omega)}
\end{equation}
for all $u \in H^1(\Omega)$. Since the interval $I$ is compact and the coefficients of $D$ are in $C(I, L^\infty(\Omega))$, we can choose $C_1$ and $C_2$ independent of $t$. The existence of $u_t$ follows from \eqref{unifcoerc} and the invertibility of $D_t$ (cf. Theorem 7.21 of \cite{F95}).

%positive constants $C_t$ and $\lambda_t$ such that
%\[
%	|D_t(u,u)| \geq C_t \|u\|^2_{H^1(\Omega)} - \lambda_t \|u\|^2_{L^2(\Omega)}
%\]
%for all $u,v \in H^1(\Omega)$. Moreover, it is immediate from the proof that $C_t$ and $\lambda_t$ only depend on the ellipticity constant of the form $D_t$ and the norms $\|b^i_t\|_{L^{\infty}(\Omega)}$, $\|c^i_t\|_{L^{\infty}(\Omega)}$, and $\|d_t\|_{L^{\infty}(\Omega)}$. Since the interval $I$ is compact there exist constants $C$ and $\lambda$ such that

We next claim that there exists $A > 0$ such that
\begin{align}
	\|u_t\|_{H^1(\Omega)} \leq A \|F_t\|_{H^1(\Omega)^*}
	\label{unifH1}
\end{align}
for any $t \in I$ and $F_t \in H^1(\Omega)^*$.
Assume this is not the case, so there exist sequences $\{t_i\}$ and $\{F_{t_i}\}$ with $\|u_{t_i}\|_{L^2(\Omega)} = 1$ and $\|u_{t_i} \|_{H^1(\Omega)} \geq i \|F_{t_i}\|_{H^1(\Omega)^*}$. The uniform coercivity bound \eqref{unifcoerc} implies $\{u_{t_i}\}$ is bounded in $H^1(\Omega)$, so there exists a subsequence with
\[
	t_i \rightarrow \bar{t}, \quad
	u_{t_i} \rightarrow \bar{u} \text{ in } L^2(\Omega) , \quad
	u_{t_i} \rightharpoonup \bar{u} \text{ in } H^1(\Omega)
\]
for some $\bar{u} \in H^1(\Omega)$ and $\bar{t} \in I$. We also have $F_{t_i} \rightarrow 0$ in $H^1(\Omega)^*$, hence $D_{\bar{t}}(\bar{u},v) = 0$ for all $v \in H^1(\Omega)$. The invertibility of $D_{\bar{t}}$ implies $\bar{u} = 0$, which is a contradiction because $\| \bar{u} \|_{L^2(\Omega)} = 1$, so the proof of \eqref{unifH1} is complete.

We are now ready to prove continuity of $t \mapsto u_t$ in $H^1(\Omega)$. It suffices to check at a single point, say $t=0$, which we can assume is contained in $I$ by performing a translation. From the definition of $u_t$ we obtain
\[
	D_0(u_t-u_0,v) = (F_t - F_0)(v) - (D_t - D_0)(u_t,v).
\]
Defining a functional $G_t \in H^1(\Omega)^*$ by $G_t(v) = (F_t - F_0)(v) - (D_t - D_0)(u_t,v)$, we have from \eqref{unifH1} that
\[
	\|u_t - u_0\|_{H^1(\Omega)} \leq A \|G_t\|_{H^1(\Omega)^*}.
\]
Since $\|u_t\|_{H^1(\Omega)}$ is uniformly bounded for $t$ close to zero (again using \eqref{unifH1}), the continuity of $D_t$ and $F_t$ implies $\|G_t\|_{H^1(\Omega)^*} \rightarrow 0$ as $t \rightarrow 0$. This completes the proof for $k=0$.

Now assume the result holds for some $k \geq 0$, and suppose that $D_t$ and $F_t$ are of class $C^{k+1}$. Differentiating the equation $D_t(u_t,v) = F_t(v)$ with respect to $t$, we obtain
\begin{align}
	D_t(u^{(k)}_t,v) = F_t^{(k)} - \sum_{j=0}^{k-1} \binom{k}{j} D_t^{(k-j)} \left( u_t^{(j)}, v \right)
	\label{utderiv}
\end{align}
For convenience we let $u^{(j)}_0$ denote the $j$th derivative of $u_t$ evaluated at $t=0$, and similarly for $D_t$. We claim that $u_0^{(k+1)}$ exists, and is equal to the unique function $w \in H^1(\Omega)$ that satisfies
\[
	D_0(w,v) = F^{(k+1)}_0(v) - D^{(1)}(u_0^{(k)},v)  - \sum_{j=0}^{k-1} \binom{k}{j} \left[ D_0^{(k+1-j)} \left( u_0^{(j)}, v \right) + D_0^{(k-j)} \left( u_0^{(j+1)}, v \right) \right]
\]
for all $v \in H^1(\Omega)$. Using \eqref{utderiv}, we find
\begin{align*}
	D_0\left(w - \frac{u_t^{(k)} - u_0^{(k)}}{t} , v\right) 
	= &\left( F_0^{(k+1)} - \frac{F_t^{(k)} - F_0^{(k)}}{t} \right)(v) \\
	+& \sum_{j=0}^k \left[ \frac{D_t^{(k-j)}(u_t^{(j)},v) - D_0^{(k-j)}(u_0^{(j)},v)}{t} \right.  - D_0^{(k+1-j)}(u_0^{(j)},v) - D_0^{(k-j)}(u_0^{(j+1)},v) \Bigg].
\end{align*}
Since $D_t$ and $F_t$ are of class $C^{k+1}$, the right-hand side defines a linear functional $H_t(v)$ with $\|H_t\|_{H^1(\Omega)^*} \rightarrow 0$ as $t \rightarrow 0$, and the uniform estimate \eqref{unifH1} yields
\[
	\lim_{t=0} \left\| w - \frac{u_t^{(k)} - u_0^{(k)}}{t} \right\|_{H^1(\Omega)} =0
\]
as was to be shown. The continuity of $u_t^{(k+1)}$ is proved in a similar fashion.
\end{proof}

The assumption of invertibility on an interval $I$ (as opposed to at a single point) is not unreasonable.

\begin{lemma} \label{openinv}
Let $\{D_t\}$ satisfy the regularity assumptions of Proposition \ref{Dform} with $k=0$. If $D_{t_0}$ is invertible for $t_0 \in I$, then $D_t$ is invertible for $|t-t_0| \ll 1$.
\end{lemma}

\begin{proof} It suffices to consider $t_0 = 0$. Suppose the claimed result is false; then there exist numbers $t_i \rightarrow 0$ and functions $u_i \in H^1(\Omega)$ such that $D_{t_i}(u_i,v) = 0$ and $\| u_i\|_{L^2(\Omega)} = 1$. The uniform coercivity estimate \eqref{unifcoerc} implies $\{u_i\}$ is bounded in $H^1(\Omega)$, so there exists a subsequence $\{u_i\}$ and a function $\bar{u} \in H^1(\Omega)$ such that $u_i \rightharpoonup \bar{u}$ in $H^1(\Omega)$ and $u_i \rightarrow \bar{u}$ in $L^2(\Omega)$. It follows that $\|\bar{u}\|_{L^2(\Omega)} = 1$. Since the coefficients of $D_t$ are continuous, and weakly convergent subsequences are bounded, we have
\begin{align*}
	D_0(\bar{u},v) = \lim_{i \rightarrow \infty} D_{t_i}(u_i,v) = 0
\end{align*}
for each $v \in H^1(\Omega)$. The invertibility of $D_0$ yields $\bar{u} = 0$, which is not possible.
\end{proof}

% new!
To apply Proposition \ref{Dform} in practice we use the following elementary consequence of the chain rule.

\begin{lemma}\label{lemmaDtreg}
Let \[
	D(u,v) = \int_{\Omega} \left[ a^{ij}(\partial_i u)(\partial_j v) + b^i (\partial_i u)v + c^i u (\partial_i v) + d uv \right] 
\]
and define $D_t$ by \eqref{Dtdef}. If the coefficients $a^{ij}, b^i, c^i, d$ are contained in $C^k\left( \overline{\cup_{a \leq t \leq b} \Omega_t} \right)$
and the family $\{\varphi_t\}$ is of class $C^k$, then the coefficients of $D_t$ are contained in $C^k(I,L^\infty(\Omega))$.
\end{lemma}

\begin{proof}
Recalling that $D\varphi_t$ denotes the Jacobian matrix of $\varphi_t$, and $|D\varphi_t|$ its determinant, we have
\begin{align*}
	D_t(u,v) &= \int_{\Omega_t} \left[ a^{ij}(\partial_i u_t)(\partial_j v_t) + b^i (\partial_i u_t)v_t + c^i u_t (\partial_i v_t) + d u_t v_t \right]  \\
	&= \int_\Omega |D\varphi_t|  \left[ a^{ij}(\partial_i u_t)(\partial_j v_t) + b^i (\partial_i u_t)v_t + c^i u_t (\partial_i v_t) + d u_t v_t \right] \circ \varphi_t 
\end{align*}
where we have defined $u_t = u \circ \varphi_t^{-1}$ and similarly for $v_t$. Therefore, the rescaled coefficients on $\Omega$ are
\begin{align*}
	a^{ij}_t = |D\varphi_t| (a^{pq} \circ \varphi_t) (D\varphi_t)^{-1}_{ip} (D\varphi_t)^{-1}_{jq}, &\quad
	b^i_t = |D\varphi_t| (b^p \circ \varphi_t) (D\varphi_t)^{-1}_{ip}, \\
	c^i_t = |D\varphi_t| (c^p \circ \varphi_t) (D\varphi_t)^{-1}_{ip}, &\quad
	d_t = |D\varphi_t| (d \circ \varphi_t),
\end{align*}
where $(D\varphi_t)^{-1}_{ip}$ denotes the $i,p$ entry of the matrix $D\varphi_t$. The hypotheses imply that $|D\varphi_t|$ and $(D\varphi_t)^{-1}_{ip}$ are contained in $C^k(I,L^\infty(\Omega))$. Since $d$ is uniformly continuous on each $\Omega_t$ and $t \mapsto \varphi_t$ is continuous, the composition $d \circ \varphi_t$ defines a continuous map from $I$ into $L^\infty(\Omega)$, and similarly for $a^{pq} \circ \varphi_t$ etc. For $k=1$ we have that
\[
	\frac{d}{dt} (d \circ \varphi_t) = \frac{d\varphi_t}{dt} \cdot (\nabla d) \circ \varphi_t
\]
is a continuous map from $I$ to $L^\infty(\Omega)$. Higher derivatives are treated in a similar fashion.
\end{proof}

\bibliographystyle{plain}
\bibliography{maslov}

\end{document}